\documentclass[10pt]{article}

\usepackage{amsthm} 
\usepackage{amsfonts, mathrsfs,amssymb}
\usepackage{amsmath}
\usepackage[numbers, sort&compress]{natbib}
\usepackage{graphicx}
\usepackage{vmargin, enumerate}
\usepackage{setspace}
\usepackage{paralist}
\usepackage[pdftex]{hyperref}

\usepackage[small,bf]{caption2}

\newcommand{\R}{\mathbb{R}}
\newcommand{\Z}{\mathbb{Z}}
\newcommand{\N}{{\mathbb N}}

\newcommand{\E}[1]{{\mathbf E}\left[#1\right]}
\newcommand{\Ec}[1]{{\mathbf E}[#1]}
\newcommand{\e}{{\mathbf E}}
\newcommand{\V}[1]{{\mathbf{Var}}\left[#1\right]}

\newcommand{\p}[1]{{\mathbf P}\left(#1\right)}  
\newcommand{\pc}[1]{{\mathbf P}(#1)}  % size-constrained

\newcommand{\I}[1]{{\mathbf 1}_{[#1]}}
\newcommand{\set}[1]{\left( #1 \right)}
\newcommand{\Cprob}[2]{\mathbf{P}\set{\left. #1 \; \right| \; #2}}
\newcommand{\probC}[2]{\mathbf{P}\set{#1 \; \left|  \; #2 \right. }}

\newcommand{\pran}[1]{\left(#1\right)}

\newtheorem{thm}{Theorem}
\newtheorem{lem}[thm]{Lemma}

\newtheorem{cor}[thm]{Corollary}

\newcommand{\ssn}{\mathcal{N}}

\newcommand{\event}[1]{\left[\, #1 \,\right]}
\newcommand{\CExp}[2]{\mathbf{E}\event{\left. #1 \; \right| \; #2}}

\hypersetup{
    bookmarks=true,         % show bookmarks bar?
    unicode=false,          % non-Latin characters in Acrobat’s bookmarks
    pdftoolbar=true,        % show Acrobat’s toolbar?
    pdfmenubar=true,        % show Acrobat’s menu?
    pdffitwindow=true,      % page fit to window when opened
    pdftitle={My title},    % title
    pdfauthor={Author},     % author
    pdfsubject={Subject},   % subject of the document
    pdfnewwindow=true,      % links in new window
    pdfkeywords={keywords}, % list of keywords
    colorlinks=true,       % false: boxed links; true: colored links
    linkcolor=blue,          % color of internal links
    citecolor=blue,        % color of links to bibliography
    filecolor=blue,      % color of file links
    urlcolor=blue           % color of external links
}

\newcommand{\phat}[1]{\ensuremath{\hat{\mathbf P}}\left(#1\right)}  
\newcommand{\phatc}[1]{\ensuremath{\hat{\mathbf P}}(#1)}            
     
\newcommand{\Cphatc}[2]{\hat{\mathbf{P}}( #1 \; | \; #2)}

\newcommand{\m}{M} % max in the living particles
\newcommand{\mb}{R} % max in the entire BRW

\begin{document}

\title{\bf Total Progeny in Killed Branching Random Walk}
\author{L. Addario-Berry \and  N. Broutin} 

\maketitle
\begin{abstract}We consider a branching random walk for which the maximum position of a particle in the n'th generation, $\mb_n$, has zero speed on the linear scale: $\mb_n/n \to 0$ as $n\to\infty$. We further remove (``kill'') any particle whose displacement is negative, together with its entire descendence. The size $Z$ of the set of un-killed particles is almost surely finite \cite{gantert2006,hu07minimal}. In this paper, we confirm a conjecture of Aldous \cite{aldouskilled,Aldous1998b} that $\E{Z}<\infty$ while $\E{Z\log Z}=\infty$. The proofs rely on precise large deviations estimates and ballot theorem-style results for the sample paths of random walks.
\end{abstract} 
%\tableofcontents

\section{Introduction}\label{sec:intro} 
%!TEX root = killedbrw.tex

Consider a branching random walk. The particles form the set of individuals of a Galton--Watson process: there is a unique ancestor (root of the tree) which gives birth to $B$ children in the first generation. The children behave independently and themselves give birth to children according to the same offspring distribution $B$. We suppose throughout the paper that this branching process is supercritical $\e B>1$, so that it survives with positive probability \cite{AtNe1972}, and that $\e B < \infty$. We can think of the set of \emph{potential} individuals as a subset of the infinite tree 
$$\mathcal U=\bigcup_{n\ge 0} \N^n,$$
where every node at level $n$ is a word $u=u_1 u_2 \dots u_n$ of $n$ integers. The root is then $\varnothing$, with potential children $1, 2,\dots$, and the structure of the tree is such that the ancestors of a node $u=u_1u_2\dots u_n$ are the prefixes $\varnothing$, $u_1,u_1u_2$, etcetera, up to $u_1\dots u_{n-1}$. Given $\{B_u, u\in \mathcal U\}$ a family of independent and identically distributed (i.i.d.) random copies of $B$, the Galton--Watson tree $\mathcal T$ is the subtree of $\mathcal U$ 
consisting of all nodes $u=u_1,\ldots,u_n$ for which, for all $1 \leq i \leq n$, $u_i < B_{u_1\dots u_{i-1}}$ 
(in the case $i=1$ this notation means that $u_1 < B_{\emptyset}$) --- see \cite{Neveu1986,LeGall2005}

We also suppose that each node $u\in \mathcal U$ carries a real position, or displacement. 
Given a family $\{X_u, u\in \mathcal U\setminus\{\emptyset\}\}$,
a family of i.i.d.\ copies of a random variable $X$, the displacement of a node $u$ is 
$$S_u=\sum_{w\preceq u,w\neq\emptyset} X_u,$$ 
where $w\preceq u$ means that $w$ is an ancestor of $u$ (and $u\preceq u$). Thus, 
for each node $u$, $X_u$ is the displacement of $u$ relative to its parent, and we let the root have 
displacement $S_{\emptyset}=0$. 
Then, along each branch of the tree $\mathcal U$, the positions of the particles follow a random walk with step size $X$. The collection $\{S_u: u\in \mathcal T\}$ is a branching random walk with step size $X$ and branch factor $B$.

We say that a particle $u\in \mathcal U$ is {\em living} if the random walk on the branch from the root to $u$ never takes a negative value: that is, $u$ is living if $$S_w\ge 0 \qquad \mbox{for all}\quad w\preceq u.$$ 
We are interested in the subtree $\mathscr L$ of $\mathcal T$ consisting only of living particles.
We say that the pair $\mathscr L$, $\{S_u:u\in \mathscr L\}$ is a \emph{killed branching random walk}. It is natural that the behaviour of the tree $\mathscr L$, and in particular its size, should be related to the behaviour of the maximum $\mb_n$ of the positions of the particles $u\in \mathcal T_n=\{x\in \mathcal \N^n: x \in \mathcal T\}$ lying in the $n$'th generation of $\mathcal T$, and we now elaborate on this. Let $\Lambda$ be the cumulant generating function of $X$ given by
$$\Lambda(\lambda)=\log \E{e^{\lambda X}},$$
defined for $\lambda \in \mathcal D$, the set of values $\lambda$ for which $\Lambda(\lambda)<\infty$. Let $\mathcal D^o$ be the interior of $\mathcal D$. Let also $f(\lambda)=\lambda \Lambda'(\lambda)-\Lambda(\lambda)$. The function $\Lambda$ is infinitely differentiable and convex on $\mathcal D^o$, and $f$ is strictly convex on $\mathcal D^o$ (see \cite{DeZe1998}, Lemma~2.2.5 and Exercise~2.2.24). 

Suppose that there exists a (necessarily unique) $\lambda \in \mathcal D^o$ with $\lambda > 0$, for which 
$f(\lambda) = \log \E{B}$. Then 
the Biggins--Hammersley--Kingman theorem \cite{Biggins1977,Kingman1975,Hammersley1974} states that, conditioned on non-extinction, the maximum position $\mb_n$ of a particle in the $n$'th generation of a well-controlled branching random walk
satisfies 
\[ 
\lim_{n\rightarrow \infty} \frac{\mb_n}{n} = \Lambda'(\lambda) \quad \mbox{ almost surely and in expectation}. 
\] 
We call the branching random walk \emph{well-controlled} if there exists a (necessarily unique) $\lambda^\star \in \mathcal D^o$ with $\lambda^\star > 0$ such that $\Lambda'(\lambda^\star)=0$. 
If the branching random walk is well-controlled then we call both the un-killed and killed branching random walks  supercritical, critical, or subcritical according as $f(\lambda^\star)$ is greater than, equal to, or less than $\log \E{B}$. 

If the killed branching random walk is supercritical, it is not hard to see that the maximum position of a living particle still tends to $+\infty$ almost surely and in expectation with the same linear speed as before. In this case, it is fairly straightforward to calculate the logarithmic growth rate of the total number of progeny in the $n$'th generation. 

In subcritical case, it is equally clear that extinction eventually occurs. 
The critical case is not as clear since it might be the case that $\mb_n=o(n)$ but that $\mb_n \to +\infty$.
However, it is not too hard to convince oneself that $\mb_n\to -\infty$ in expectation, since $\e{\mb_{kn}} \geq k\e{\mb_n}$ for all $k,n \geq 1$. 
(This can be seen by considering first the particle of $v$ of maximal displacement at the $n$'th generation, then the 
particle of maximal displacement at the $2n$'th generation that is a descendent of $v$, and so on.) If $\e{\mb_n}$ were positive for some $n$, it 
would then follow that at least along a subsequence, $\e{\mb_{n}}$ would grow at a positive linear speed, 
contradicting the Biggins--Hammersley--Kingman theorem. In fact, in the critical case, \citet{hu07minimal} have proved that almost surely 
\begin{equation}\label{eq:hu-shi}
\limsup_{n\to \infty} \frac{\mb_n}{\log n} \le - \beta,
\end{equation}
for some positive constant $\beta$, which implies that eventually every branch goes extinct with probability one. (In the special case that $X \in \Z$ a.s.,
 this also follows from work of \citet{gantert2006}.)

In these last two cases, the parameters of interest are the total number $Z=|\mathscr L|$ of living individuals in the process, the maximum location that a particle ever reaches, 
$$\m = \sup_{n \geq 0} \m_n,$$
where $\m_n = \sup\{S_u : u \in \mathscr L_n\}$, 
and where $\mathscr L_n=\mathscr L \cap \N^n$.
\citet{aldouskilled} has conjectured that in the critical case, $\e{Z}< \infty$ but $\E{Z\log Z}=\infty$, and that in the subcritical case, $Z$ has power law tails. 

\citet{pemantle99killed} found exact asymptotics for the probability distribution of $Z$ in the instructive special case that $X\in \{-1, +1\}$. In this setting, the criticality condition implies that $p:=\p{X=1}=(2-\sqrt 3)/4$, the smallest root of $16p(1-p)=1$. He found that 
\[\p{Z=n} = \frac{c +o(1)}{n^2 \log^2 n}, \qquad \mbox{with}\qquad c= \frac{\log(1/4p)}{4p}=4.915\dots.\]
It is then clear that $\E{Z}<\infty$ while $\E{Z\log Z}=\infty$. His proof relies on a recursive description of the process. The fact that $X$ only takes unit steps turns out to be crucial and allows for a precise study of the probability generating function $\E{s^Z}$ via singularity analysis methods \cite{FlOd1990,FlSe2009}. 
          
In this paper we verify the critical case of Aldous' conjecture. 

\begin{thm}\label{thm:critical}Consider a critical killed branching random walk and let $Z$ be the total progeny of the process. Then $\E{Z} < \infty$. If additionally $\E{B\log^8 B} < \infty$ then $\E{Z\log Z}=\infty$. 
\end{thm}

\paragraph{Remarks.} 
The moment condition on $B$ that arises in the above theorem is technical and is required for the use of 
the size-biasing technique explained below. We believe that the theorem should hold as long as $\e{B}<\infty$. 
We were not able to obtain more detailed information about the probability distribution of $Z$. 
(Our approach can provide upper bounds on the tail probabilities of $Z$, via Markov's inequality, 
but does not seem well-suited to proving lower bounds for such tail probabilities in either the critical or subcritical case.) 
However, it is very likely the case that for a large class of critical killed branching random walks, 
\[\p{Z=n}=\Theta\pran{\frac 1{n^2 \log^2 n}}.\]

We also provide the following estimates for the maximum position of any living particle. 

\begin{thm}\label{thm:max}The maximum position $\m$ attained by any particle in a critical killed branching random walk 
satisfies
$$\p{\m\ge k}\le e^{-k\lambda^\star }\quad \mbox{for all}~k,\qquad \mbox{and}\qquad\p{\m = k}=\Omega\pran{\frac {e^{-k \lambda^\star }}{k}}\quad \mbox{as}~k\to\infty.
$$ 
\end{thm}

Our approach in this document is rather orthogonal to the recursive one used by \citet{pemantle99killed}: using large deviations estimates for sums of i.i.d.\ random variables, we analyze the shape of the random walks along the branches of the process. This technique was also used in \citet{AdRe2007} to precisely analyse  minima in branching random walks. 
The large deviations estimates we require are stated in Theorem \ref{thm:ld_badahur}, below, and can be found in 
%main basis for our 
\cite{BaRa1960,DeZe1998}.
     
\medskip
\noindent\textsc{A notational interlude and an aside on size-biasing.}
For any tree $T$, deterministic or random, we write $T_n$ for the set of nodes of $T$ in the $n$'th generation. 
We use the notation  $T_{\leq n}$ for both $\bigcup_{i=0}^n T_i$ and for the subtree of $T$ on this set of nodes. (The ambiguity 
in the notation will always disappear in context.) 

The size-biasing technique introduced by \citet{KaPe1976} ---and used to study branching random walks in, for instance, \cite{ChRo1988,LyPePe95a,Lyons1997, BiKy2004}--- allows to formally pick a typical node in the $n$'th generation of a tree, and will be very useful in our calculations.  
We write $\hat{\mathcal T}$ for the {\em size-biased} version of $\mathcal T$, grown as follows. 
Let $\hat B$ be the size-biased version of $B$, with distribution defined by 
\[
\pc{\hat B = k} = \frac{ k \p{B=k}}{\e B}.
\]
Let $v_0$ be the root of $\hat{\mathcal T}$ and let $v_0$ have a random number of children chosen according to $\hat B$. 
Choose a child of $v_0$ uniformly at random ---say $v_1$. From all {\em other} children, 
grow independent branching processes with unbiased offspring distribution $B$. From $v_1$, independently produce a size-biased number of children, 
choose a uniform child to size-bias, and repeat ad infinitum. This process always yields an infinite tree, with a single distinguished path 
$(v_0,v_1,\ldots)$, the \emph{spine}. 

Let $\mu$ (resp.~$\hat{\mu}$) be the measure of $\mathcal T$ (resp.~$\hat{\mathcal T}$), and let $\hat{\mu}^\star$ be the joint measure of 
$\hat{\mathcal T}$ and $(v_0,v_1,\ldots)$. Let $[T]_{\le n}$ be the set of trees that agree with the tree $T$ on the first $n$ levels. For $v\in T_n$, let $[T,v]_{\le n}$ be the set of trees with a distinguished path agreeing with $T$ on the first $n$ levels, and with a spine going through $v$.
\citet{LyPePe95a} show that for all $n$ and all $T$, if $T_n \neq \emptyset$ then for all $v \in T_n$, 
\[
\hat{\mu}^* [T,v]_{\le n} = \frac{1}{(\e{B})^n} \mu [T]_{\leq n} \qquad \mbox{and} \qquad \hat{\mu}[T]_{\leq n} = \frac{|T_n|}{(\e{B})^n} \mu[T]_{\leq n}.
\]
For our purposes, the ordering of the children of a node within $\mathcal T$ will always be unimportant, 
and so 
we may equivalently imagine growing $\hat{\mathcal T}$ in the following way. 
Start from an infinite path $(v_0,v_1,\ldots)$ in $\mathcal U$ ---say the ``leftmost'' path 
$\emptyset, 1, 11$, and so on--- which will form the spine.
Independently give each node $v_i$ a random number $C_i$ of children off the spine, where $C_i$ has distribution 
$\hat{B}-1$, and start an independent branching process with offspring distribution $B$ from each node off the initial infinite path. 
We write $\phat{\,\cdot\,}$ for the probability operator corresponding to $\hat{\mathcal T}$.
We also write 
$\hat{\mathscr L}$ for the subtree of $\hat{\mathcal T}$ consisting only of living particles. We refer to both $\hat{\mathcal T}$ and $\hat{\mathscr L}$ as {\em tilted} trees. 

\medskip 
\noindent\textsc{A sketch of the approach.}
We first explain how we bound the expectation $\e Z$. 
Decomposing the tree by level yields
\begin{align}\label{eq:size-bias}
\E{Z}	 = \E{|\mathscr L|}
	& = \sum_{n \ge 0} \E{|\mathscr L_n|}\nonumber  \\
	& = \sum_{n \ge 0} \sum_{T \subseteq \mathcal U_{\le n}}  \sum_{v \in T_n} \probC{v \in \mathscr L_n}{\mathcal T_{\le n}= T} \cdot \mu [T]_{\le n} \nonumber \\
	& = \sum_{n \ge 0} (\e{B})^n \sum_{T \subseteq \mathcal U_{\le n}}\sum_{v \in T_n} \probC{v \in \mathscr L_n}{\mathcal T_{\le n}= T} \cdot \hat{\mu}^\star[T,v]_{\le n} \nonumber \\
	& = \sum_{n \ge 0} (\e{B})^n \cdot \phatc{v_n \in \hat{\mathscr{L}_n}}.
\end{align}
Proving that $\E{Z}< \infty$ thus amounts to proving upper bounds on $\phatc{v_n \in \hat{\mathscr{L}_n}}$. 
Letting $\{X_i, i\ge 0\}$ be a sequence of i.i.d.\ random variables distributed like $X$, we thus seek bounds on 
the probability that the random walk defined by $S_i:=\sum_{j=1}^i X_j$, $i=0, \dots, n$ stays positive. 
Two remarks are now in order: first, since $\E{X}<0$, the event $\{S_i \ge 0, i=1, \dots, n\}$ lies in the realm of large deviations; and second, controlling the probability that $v_n \in \hat{\mathcal{T}_n}$ 
amounts to estimating ``ballot-style'' sample path probabilities. Given that $S_n\ge 0$, large deviations bounds imply that $S_n$ most likely lies around zero \cite{DeZe1998}. We are then led to estimating $\Cprob{S_i \ge 0, i=1, \dots, n}{S_n=0}$, which we will see, satisfies
\[\Cprob{S_i \ge 0, i=1, \dots, n}{S_n=0} = \Theta\pran{\frac{1}n},\]
as for mean-zero random walks \cite{Bertrand1887,AdRe2008}. 

Proving that $\E{Z \log Z}=\infty$ turns out to be harder. As for the classical $x\log x$ moment condition of the Kesten--Stigum theorem \cite{KeSt66a,LyPePe95a,AtNe1972}, the phenomenon is due to a lack a concentration of the number of particles $Z$. This is why we are led to investigate events of very low probability to find a relevant lower bound on $\E{Z\log Z}$. The events we will consider ensure that there exists a particle $v$ with high enough position that its descendence is huge: indeed, despite the negative drift of the random walk the branching property makes sure that the collection of living particles is extremely large before the drift can send all descendants of $v$ back to a negative position. The kind of ``high position'' particles we require for our proof will have displacement roughly $k=\Theta(\sqrt n)$, where $n$ is the generation of the particle. To prove lower bounds on the probability that such a particle exists, we add extra constraints which ensure the concentration of the number of such particles, then use the second moment method \cite{AlSp2008,ChEr1952}. The events we will be interested in are roughly of the form 
\[\{0<S_i<k \mbox{~for~}i=1,\dots, n, \mbox{~and~} S_n=k\},\]
and estimating their probabilities amounts to deriving sample path probability estimates for random walks with two barriers. 
 
\medskip
\noindent\textsc{Motivation and related work.}
The model arose from research on combinatorial optimization in trees. One is given a complete tree, a binary tree say, and is asked to devise an algorithm that would find the large values of $S_v$ \cite{KaPe1983,Aldous1992a,Aldous1998b,Pemantle2007a}. A natural idea for an algorithm is to intentionally not explore the subtrees rooted at nodes with too small a value, the negative ones, say. The first natural question is then that of the survival probability of the algorithm, since it might be stuck early despite the presence of nodes with large values deeper in the tree. \citet*{GaHuSh2009} settle the question about the scaling behavior of the survival probability in the near critical case, $R_n/n\to \epsilon$, as $n\to\infty$ then letting $\epsilon\to0$. The analogous continuous model of branching Brownian motion with absorption has been studied by \citet{Kesten1978} and by~\citet{HaHa2007}. For similar analysis from a statistical physics perspective, see \cite{DeSi2007,SiDe2008}. 
   
\medskip
\noindent\textsc{Plan of the paper.}
In Section~\ref{sec:bahadur-rao}, we introduce the large deviations tools we will need for in the proofs of our main results. 
In Section~\ref{sec:ballot}, we state the results we require about sample paths for random walks.
In Section~\ref{sec:maximum}, we provide upper and lower tail bounds for the maximum position ever attained in the killed and un-killed critical branching random walks. 
Section~\ref{sec:ZlogZ} is devoted to the proof of Theorem~\ref{thm:critical}: we prove that $\E{Z}<\infty$ and $\E{Z\log Z}=\infty$. 
Finally, in Section \ref{sec:proofs-ballot}, we provide the proofs of the sample paths results. The analyses in this section are based on recent work by \citet{AdRe2008}. 
                 
\section{Precise large deviations}\label{sec:bahadur-rao}
%!TEX root = killedbrw.tex

Before going any further, we establish one assumption to which we adhere for the duration of the paper. 
We say $X$ is a {\em lattice random variable} with period $d>0$ 
if there is a constant $z \in \R$ such that $dX-z$ is almost surely integer-valued, and $d$ is the smallest positive real number 
for which this holds; in this case, we say that the set $\mathbb{L}_X=\{(n+z)/d:n \in \Z\}$ is {\em the lattice of} $X$. 
Technically, the analysis of the paper should have two cases, depending on whether or not $X$ is a lattice random variable. 
However, these cases are essentially identical, and the formulae are shorter for lattice random variables. 
{\em We thus assume from this point forward that the step size $X$ is a lattice random variable with lattice~$\Z$.} 

As in the introduction, we define the {\em logarithmic moment generating function} 
\[
\Lambda(t) = \Lambda_X(t) := \log \E{e^{tX}},
\]
and usually supress the $X$ in the subscript since it will be clear from context.
To better understand the utility of the function $\Lambda$ in deriving tail bounds, we first recall Chernoff's 
bounding technique \cite{chernoff52measure}. If $S_n = \sum_{i=1}^n X_i$ is a sum of $n$ independent copies of $X$, then for any
 $c > \E{X}$ and $t > 0$, by using Markov's inequality and independence, we have 
\begin{equation}
\p{S_n \geq cn} = \p{e^{tS_n} > e^{tcn}} \leq  \frac{\E{e^{tS_n}}}{e^{tcn}} = \pran{\E{e^{t(X-c)}}}^n =  e^{-n(tc-\Lambda(t))},\nonumber
\end{equation}
by definition of $\Lambda(t)$. We choose the value of $c$ that minimizes this upper bound: 
\begin{equation}\label{eq:ld_intuitive}
  \p{S_n \geq cn} \leq  \exp\pran{-n \sup_{t>0}\{tc -\Lambda(t)\}}.
\end{equation}
The optimal choice for $t$ in (\ref{eq:ld_intuitive}) is then that for which $\Lambda'(t)=c$ ---if such a $t$ exists--- as may be informally seen by differentiating $t\mapsto tc-\Lambda(t)$ with respect to $t$. Choosing $t$ in this fashion and writing 
$\Lambda'(t)$ in place of $c$ yields   
\begin{equation}\label{eq:ld_almosttight}
\p{S_n \geq \Lambda'(t)n} 
\le e^{-n(t\Lambda'(t)-\Lambda(t))}.  
\end{equation}
It turns out that the upper bound given by (\ref{eq:ld_almosttight}) is almost tight; this is the substance of 
the ``exact asymptotics for large deviations'' first proved by \citet{BaRa1960}, and is the reason that 
the behavior of $\Lambda$ is key to our investigation. We now formally introduce this result.

We now state a version of asymptotic estimates for large deviations, essentially appearing in \citet{DeZe1998} (as Theorem 3.7.4). The following notation will be convenient: for a parameter $C$, we write $O_C(\,\cdot\,)$, $\Omega_C(\,\cdot\,)$ and $\Theta_C(\,\cdot\,)$ to emphasize that the constants hidden in the asymptotic terms depend on~$C$.
\begin{thm}[\citet{BaRa1960}]\label{thm:ld_badahur}
Let $S=\{S_n\}_{n \in \N}$ be a random walk with integer step size $X$, and define $\Lambda$ and $\mathcal D_{\Lambda}^o$ as above. 
Choose any $\lambda \in \mathcal D_{\Lambda}^o$ with $\lambda > 0$ and any constant $C > 0$. 
Then for any $a \in \Z$ with $|a| \leq C\sqrt n$, 
\begin{equation}\label{eq:ld_nonlat}
  \p{S_n = \Lambda'(\lambda)n+a} = \Theta_C(1) \cdot\frac{e^{-a\lambda-nf(\lambda)}}{\sqrt{\Lambda''(\lambda)\cdot 2\pi n}} = \p{S_n \geq \Lambda'(\lambda)n+a}.
\end{equation} 
\end{thm}   
This theorem is stated with $a$ constant in \cite{DeZe1998}, but a straightforward 
modification yields the above formulation. (See also \cite{Petrov1965} and \cite[][Chapter VIII, p.\ 248]{petrov75sums} for an even stronger, uniform version of this result, stated in slightly different language.) 
A similar result holds in the non-lattice case, if we replace the event $\{S_n=\Lambda'(\lambda)n+a\}$ with 
$\{(S_n-\Lambda'(\lambda)n-a) \in [0,c]\}$ for an arbitrary fixed positive constant $c$. This is the version we would use if we were to explicitly treat the non-lattice case.

The proof of Theorem~\ref{thm:ld_badahur} consists in an exponential change of measure (in order to be able to work with centered random variables) combined with the Berry--Ess\'een extension of the central limit theorem \cite{Berry1941,Esseen1963, Feller1971}. 
The same change of measure will be more generally useful to us, and we take a moment to explain it in detail and derive some easy consequences. 

Let $F$ be the distribution function of $X$. We remark that for $\lambda \in \mathcal D^o$, 
\[
\Lambda'(\lambda) = \frac{\E{Xe^{\lambda X}}}{\E{e^{\lambda X}}}
\qquad
\mbox{ and }
\qquad \Lambda''(\lambda) = \frac{\E{X^2e^{\lambda X}}}{\E{e^{\lambda X}}}-
	\bigg(\frac{\E{Xe^{\lambda X}}}{\E{e^{\lambda X}}}\bigg)^2.
\]
Consider the random variable $Y_\lambda$, with distribution function $G_\lambda$ defined by  
\[
G_{\lambda}(x) = \frac{1}{\E{e^{\lambda X}}} \int_{-\infty}^x e^{\lambda x}\,dF(x).
\]
Then, we have 
\begin{align}\label{eq:y_exp}
\E{Y_{\lambda}} = \int_{-\infty}^{\infty} x\,dG_{\lambda}(x) &= \frac{1}{\E{e^{\lambda X}}} \int_{-\infty}^{\infty} xe^{\lambda x}\,dF(x) = \Lambda'(\lambda),  \\ 
\E{Y_{\lambda}^2} = \int_{-\infty}^{\infty} x^2\,dG_{\lambda}(x)&= \frac{1}{\E{e^{\lambda X}}}\int_{-\infty}^{\infty} x^2e^{\lambda x}\,dF(x) = \Lambda''(\lambda)+\Lambda'(\lambda)^2,  \nonumber
\end{align}
so the random variable $Z_\lambda=Y_\lambda-\Lambda'(\lambda)$ is such that $\E{Z_\lambda}=0$ and $\V{Z_{\lambda}} = \V{Y_\lambda}=\Lambda''(\lambda)$. 
It may also easily be checked that $\E{|Y_{\lambda}|^3} < \infty$, a fact we will use later. 

We now show that we may express the probability of events such as $\{X_1+\dots +X_n \ge c n\}$, which belong to the large deviations regime for $c>\E{X}$, in terms of the distribution of the sum $Z_1+\dots+Z_n$ of i.i.d.\ copies of $Z_\lambda$ in the central regime. This allows for the use of precise limit results related to the central limit theorem.

Let $S_n = X_1+\ldots+X_n$ and fix any $a\in \R$. Then for any $c$ and any $\lambda \in \mathcal D_{\Lambda}^o$, using the same change of measure as in (\ref{eq:y_exp}), we have  
\begin{eqnarray}
	\p{S_n\geq cn+a} 	& = & \int_{\{x_1+\dots+x_n \geq cn+a\}}\,dF(x_1)\cdots dF(x_n) \nonumber\\
					& = & e^{n\Lambda(\lambda)} \int_{\{y_1+\ldots+y_n\geq tn+a\}} e^{-\lambda(y_1+\ldots+y_n)}\,dG_{\lambda}(y_1)\cdots dG_{\lambda}(y_n). \nonumber
\end{eqnarray}
The centered random variable $Z_\lambda$ has distribution function $H_\lambda$ satisfying $dH_\lambda(z) = e^{\lambda \Lambda'(\lambda)} dG_\lambda(z)$. 
So, taking $c = \Lambda'(\lambda)$, this further change of measure yields 
\begin{align*}
\p{S_n\geq \Lambda'(\lambda)n+a}  
& =  e^{n\Lambda(\lambda)} \int_{\{z_1+\ldots+z_n \geq a\}} e^{-\lambda(z_1+\ldots+z_n)}e^{-\lambda \Lambda'(\lambda)n}\,dH_{\lambda}(z_1)\cdots dH_{\lambda}(z_n) \nonumber\\
& =  e^{-nf(\lambda) } \int_{\{z_1+\ldots+z_n \geq a\}} e^{-\lambda(z_1+\ldots+z_n)}\,dH_{\lambda}(z_1)\cdots dH_{\lambda}(z_n).
\end{align*} 
Writing $W_n$ for the distribution function of $Z_1+\ldots+Z_n$, $n$ i.i.d.\ copies of $Z_\lambda$, the preceding equation asserts that 
\begin{equation}\label{eq:useful}
\p{S_n\geq \Lambda'(\lambda)n+a} = e^{-nf(\lambda) } \int_{-\infty}^{\infty} e^{-\lambda s} \I{s \geq a} \,dW_n(s).
\end{equation}
To prove the Bahadur--Rao theorem from Equation (\ref{eq:useful}) is a matter of an integration by parts, followed by an application of the Berry--Ess\'een bound between the distribution of a rescaled sum of i.i.d.\ random variables and a normal \cite{Berry1941, Esseen1963, Feller1971}. (See \cite{DeZe1998} for details.)
The useful thing about the chain of argument leading to (\ref{eq:useful}) is that we may apply it when studying other events than $\{S_n \geq \Lambda'(\lambda) n + a\}$. This is especially useful when $\Lambda'(\lambda)=0$, i.e. when $\lambda=\lambda^\star$, since in this case, the events we are considering on the left and right-hand side are identical. In particular, we shall use the following lemma to transfer the results of the next section, about sample paths of centered random walks, into the large deviation regime. 

\begin{lem}\label{lem:transfer}Let $\tilde S_n=\sum_{i=1}^n Z_i$, where $Z_1,\dots, Z_n$ are i.i.d.\ copies of the centered random variable $Z_{\lambda^\star}$ described above. Then, for all integer $n$ and Borel set $\mathscr B \subseteq \R^n$, 
$$\p{(S_1,\ldots,S_n) \in \mathscr B}=e^{-nf(\lambda^\star)} \cdot \E{\exp(-\lambda^\star \cdot \tilde S_n) \cdot \I{(\tilde S_1,\ldots,\tilde S_n) \in \mathscr B}}.$$
\end{lem}

The proof of Lemma \ref{lem:transfer} consists in mimicking the argument leading to (\ref{eq:useful}) and we omit it. 
To conclude this section, we state the useful inequality of \citet{chernoff52measure}, a weaker but simpler version of the above asymptotic result. 
\begin{lem}[Chernoff bound]\label{lem:chernoff}
For all positive $\lambda \in \mathcal D^o$, for all $a> 0$ and all integers $n\geq 1$, 
\[\p{S_{n}\geq \Lambda'(\lambda)n+a} \leq e^{-n(\lambda \Lambda'(\lambda)-\Lambda(\lambda ))-a\lambda}=e^{-n f(\lambda)-a \lambda}.\]
\end{lem}
                   
\section{The shape of random walks}\label{sec:ballot}
%!TEX root = killedbrw.tex

In this section, we collect the facts about sample path probabilities that we will require for the proofs of the main results. 
Throughout the section, $X$ is a random variable with lattice $\Z$, with $\E{X}=0$ and $0 < \E{X^2} < \infty$, 
and $S$ is a simple random walk with step size $X$. 
\citet{AdRe2008} proved the following theorem. 
\begin{thm}\label{thm:ballot_mean0}
	Fix $c > 0$. 
	Then for all $n$ and for all $k$ and $m$ with $0 < k \leq c\sqrt{n}$ and $0 \le m \leq c\sqrt{n}$, 
	\[\p{S_n=k, S_i \ge -m ~\forall~ 0 < i < n} = \Theta_{c}\pran{\frac{(m+1)(k+m+1)}{n^{3/2}}}.\]
\end{thm}
In fact, in \citep{AdRe2008} the theorem was only stated with $m=0$ but an essentially identical proof yields the above formulation.
The following theorem strengthens Theorem~\ref{thm:ballot_mean0}, under the additional assumption that $\E{|X|^3} < \infty$. 
It essentially says that an upper barrier lying $\Theta(\sqrt n)$ above the ending height of the conditioned path does not significantly constrain 
the walk. (The assumption that $\E{|X|^3} < \infty$ is not in fact necessary for any of the below theorems and corollaries, 
but it simplifies the proofs, and it will hold for the random walks to which we apply the results since they are well-controlled.)
\begin{thm}\label{thm:gbt_new}
	Fix $c > 0$ and $\epsilon > 0$. If $\E{|X|^3} <\infty$ then for all $n$ and all $k$ and $m$ with 
	$0 \leq m \leq c\sqrt{n}$ and $-m \leq k \leq c\sqrt{n}$, 
	\[\p{S_n=k, -m \leq S_i \le \max\{k,0\}+\epsilon\sqrt{n}~\forall~ 0 < i < n}  = \Theta_{c,\epsilon}\left(\frac{(k+m+1)(m+1)}{n^{3/2}}\right).\]
\end{thm}
From Theorems~\ref{thm:ballot_mean0} and~\ref{thm:gbt_new}, the key bounds we require later in the paper follow straightforwardly. 
\begin{cor}\label{fnk}
Fix $c > 1$. If $\E{|X|^3}< \infty$, then for any $n$ and all $k$ with $c^{-1} \le k/\sqrt{n} \le c$, 
\[
\p{S_n=k~;~0\le S_i<k~\forall~0<i<n} = \Theta_c\left(\frac{k+1}{n^{2}}\right).
\]
\end{cor}
For reasons that will become clear in Section~\ref{sec:maximum}, we will need to further constrain the path of the walk. This can be done without significantly changing the sample path probability. 
\begin{cor}\label{gnk}Fix $c > 1$. If $\E{|X|^3}< \infty$, there is $m_0=m_0(c)$ such that
for any $n$ and all $k$ with $c^{-1} \le k/\sqrt{n} \le c$,
\[
\pc{S_n=k~;~0\le S_{n-i}\le k~\forall~0\le i<n~;~S_{n-m}<k-m^{1/7}~\forall m\ge m_0} = \Theta_c\left(\frac{k+1}{n^{2}}\right).
\]
\end{cor}
                     
\section{Asymptotics for the maximum}\label{sec:maximum}
%!TEX root = killedbrw.tex

Our main aim in this section is to prove lower bounds on the tail probabilities for $\m=\sup\{S_u: u\in \mathscr L\}$, the maximum position of a \emph{living} particle. As indicated in the introduction, this turns out to be crucial in our proof that $\E{Z\log Z}$ is infinite. The not-quite matching upper bound announced in Theorem~\ref{thm:max} rely on completely different arguments, and we shall delay its presentation until Section~\ref{sec:ZlogZ}.

We first prove a straightforward upper tail bound on $\mb=\sup\{S_u: u\in \mathcal T\}$, the maximum position of a particle (living or not) in the branching random walk, using Markov's inequality and the size-biasing technique. 
\begin{lem}\label{maxupperbound}
For a critical branching random walk, we have, 
for all $k \geq 1$, $\p{\mb \geq k} \leq e^{-\lambda^\star k}$. 
\end{lem}
\begin{proof}
For $\ell \geq 0$, write $f_\ell = \p{\mb \geq \ell}$. 
Then the sequence $\{f_{\ell}\}_{\ell \in \N}$ is supermultiplicative. 
This is straightforwardly seen, since in order to have $\mb \geq \ell+m$, it suffices to first find a node 
$v \in \mathcal T$ with $S_v \geq \ell$, then find a node $w$ in $\mathcal T_v$ with $S_w \geq m$.
It follows by Fekete's lemma \cite{Fekete1923} (see also, e.g., \cite{Steele1997}) that there exists $c \leq \infty$ such that 
\begin{equation}\label{supermult}
\lim_{\ell \to \infty} \frac{\log f_{\ell}}{\ell} = \sup_{\ell \geq 1} \frac{\log f_{\ell}}{\ell} = c.
\end{equation}
We claim that as $k\to\infty$, 
\begin{equation}\label{bootstrap}
\p{\mb\geq k}=O(k^2e^{-\lambda^\star k}).
\end{equation}
Assuming (\ref{bootstrap}), the lemma then follows immediately. To see this, 
note that if we had $c > -\lambda^\star$, then by (\ref{supermult}) there would exist $c'$ with 
$c > c' > -\lambda^\star$ and $K>0$ such that for all $\ell \geq K$, $\p{R \geq \ell} \geq e^{-c' \ell}$, 
which contradicts (\ref{bootstrap}). It thus remains to prove (\ref{bootstrap}). 

By Markov's inequality and size-biasing, we have, for $k\ge 1$,
\begin{align*}
\p{\mb\geq k}  &= \sum_{n\ge 1} \p{\mb_{i} < k~\forall 0\le i<n; \mb_n \geq k}\\
&  \leq \sum_{n\ge 1} (\e B)^n \cdot \phat{S_{v_i} < k, i=0,\ldots,n-1;~S_{v_n} \geq k}.
\end{align*}
For $n < k^2$, by Chernoff's bound (Lemma~\ref{lem:chernoff}) we have 
\[
\phat{S_{v_i} < k, i=0,\ldots,n-1;~S_{v_n} \geq k} \leq \p{S_{v_n} \ge k} \leq (\e B)^{-n} e^{-\lambda^\star k}.
\]
For $n \ge k^2$, we can apply Theorem~\ref{thm:ballot_mean0} to the reverse random walk $(S_n-S_{n-i}, i=0, \dots, n)$
and the exponential change of measure in Lemma \ref{lem:transfer} to obtain 
\[
\phat{S_{v_i} < k, i=0,\ldots,n-1,~S_{v_n} \geq k} = O\pran{\frac{k}{n^{3/2}} \frac{e^{-\lambda^\star k}}{(\e B)^n}}, 
\]
uniformly over all $n \ge k^2$. Summing these two bounds yields 
\[
\p{\mb \geq k} \leq k^2e^{-\lambda^\star k} + O \Bigg(\sum_{n\ge k^2} \frac{ke^{-\lambda^\star k}}{n^{3/2}}\Bigg) = O(k^2e^{-\lambda^\star k}).  \qedhere
\]
\end{proof}

Finding a good lower bound for the tail probabilities for $\m$ is more technical. We believe that the tail bounds of Lemma~\ref{maxlowerbound} are in fact of the correct order. 

\begin{lem}\label{maxlowerbound}For a critical killed branching random walk, if $\E{B\log^8 B} < \infty$ then 
\[
\p{\m=k} = \Omega\pran{\frac{e^{-\lambda^\star k}}{k}}.
\]
\end{lem}

Since $\m\le \mb$
Lemmas~\ref{maxupperbound} and~\ref{maxlowerbound} together already sandwich $\p{\m =k}$ in a relatively small interval. 
There are two reasons we are unable to prove matching upper bounds for Lemma~\ref{maxlowerbound}. The first is that Theorem~\ref{thm:ballot_mean0} only 
applies when $k=O(\sqrt{n})$, which one of the reasons we are required 
to split the sum in Lemma~\ref{maxupperbound}. The second is that when $n$ is much less than $k^2$, the large deviations regime of $S_{v_n}$ begins to change. 
In principle, the uniform version of Theorem~\ref{thm:ld_badahur} proved by \citet{Petrov1965,petrov75sums} could be used in this case. However, to obtain matching bounds 
it would still be necessary to exploit the fact that the random walk must remain positive ---in other words, some ballot-style sample path probability bound would still be 
needed in this regime and, as far as we are aware, no such result has been proved.

Before proving Lemma~\ref{maxlowerbound}, we introduce some relevant notation. 
Given $v \in \mathcal U$ and $w \prec v$ (i.e.~$w \preceq v$ and $w \neq v$), let $w_v$ be the first node after $w$ on the path from $w$ to $v$. 
If $w \in \mathscr L$, we define the (shifted) maximum 
\[
\m^{w} = \max\{S_x-S_w~:~x \in \mathscr L, w \preceq x\},
\] 
and we let ${\hat \m}^w$ be the equivalent in $\hat {\mathscr L}$. 
In the tilted setting, for a vertex $v_i$ along the spine, we define the \emph{maxima off the spine}, 
\begin{align} \label{eq:def_rstar}
\hat \mb^{\star v_i}& = \max\{S_x-S_{v_i}~:~x \in \hat {\mathcal T}, v_i \preceq x, v_{i+1} \not\preceq x\}   \\
\hat \m^{\star v_i}& = \max\{S_x-S_{v_i}~:~x \in \hat {\mathscr L}, v_i \preceq x, v_{i+1} \not\preceq x\},\nonumber 
\end{align}                                                  
where we set $\hat M^{\star v_i}=-\infty$ if $v_i \notin \hat {\mathscr L}$. 
For any $k$ and $n$, let $\mathcal{M}_{k,n}$ be the event that the maximum is $k$, 
the maximum is first (level-wise) achieved by a single individual $v \in \mathscr L_n$, 
and no other node $w\in \mathscr L \setminus \mathscr L^v$ has displacement $k$ or higher. 
In symbols, we have 
\begin{equation}\label{eq:def_mkn}
\mathcal{M}_{k,n}=\{\m=k,\exists~v \in \mathscr L_n~:~S_v=k,~\forall w \in \mathscr L\setminus \mathscr L^v,  S_w<k\}.
\end{equation}
Also, let 
\[
B_{k,n}=\{v \in \mathscr L_n~:~S_v=k,~\forall w \in \mathscr L\setminus \mathscr L^v, S_w < k\},
\]
so $\mathcal{M}_{k,n} = \{\m=k\} \cap \{B_{k,n}\neq \emptyset\}$. 
We will end up working with respect to the tilted tree $\hat{\mathcal{T}}$ and its pruned subtree of living nodes $\hat{\mathscr {L}}$, and make corresponding versions 
of the above events and variables ---so, for example: 
\[
\hat{B}_{k,n}=\{v \in \hat {\mathscr L_n}~:~S_v=k,~\forall w \in \hat {\mathscr L}\setminus \hat {\mathscr L}^v, S_w < k\}.
\]

The following lemma is the crucial lower bound on the probability of existence of particles with a large displacement.
\begin{lem}\label{mknlowerbound}
For a critical branching random walk, if $\E{B\log^8 B} < \infty$ then uniformly for all $k \geq 1$ and all $n$ with $k^2 \leq n \leq 2k^2$, 
	\[
	\p{\mathcal{M}_{k,n}} = \Omega\left(\frac{ke^{-\lambda^\star k}}{n^2}\right) = \Omega\left(\frac{e^{-\lambda^\star k}}{k^3}\right).
	\]
\end{lem}
We remark that Lemma~\ref{maxlowerbound} follows immediately from Lemma~\ref{mknlowerbound}; 
the events $\mathcal M_{k,n}$ are disjoint, so we have 
$$
\p{M=k} \geq \sum_{k^2\le n\le 2 k^2} \p{\mathcal{M}_{k,n}} = \Omega\pran{\frac{e^{-\lambda^\star k}}{k}}.
$$
\begin{proof}[Proof of Lemma~\ref{mknlowerbound}]
We first remark that $B_{k,n}$ is either empty or contains just one individual ---so in particular 
$\p{B_{k,n} \neq \emptyset} = \e{|B_{k,n}|}$. 
We will in fact prove that
\begin{equation}\label{bknbound}
\p{B_{k,n}\neq \emptyset} = \Omega\pran{\frac{ke^{- \lambda^\star k}}{n^2}}.
\end{equation}
From the preceding equation, the lemma follows immediately as 
\[
\p{\mathcal{M}_{k,n}} = \probC{M=k}{B_{k,n} \neq \emptyset} \cdot \p{B_{k,n} \neq \emptyset}. 
\]
However, if $B_{k,n} \neq \emptyset$ then there is a single individual $v \in \mathscr {L}_n$ with $S_v=k$, 
and for all $w \prec v$, $S_w < k$ and no descendent $u$ of $w$ such that $u\not\in \mathscr L^v$ has $S_u \geq k$. Thus, 
if $B_{k,n} \neq \emptyset$ then for $\{M=k\}$ to occur it suffices that the node $v$ has no descendants $u \in \mathscr L$ with $S_u > S_v$. 
The probability of the latter event is bounded below by the probability that of $\sup\{S_u:u \in \mathcal T^v\}\le S_v$, i.e., that $\mb^v=0$. Since in the critical branching random walk $\mb\to-\infty$ a.s.\ by (\ref{eq:hu-shi}), we have
$\p{R=0} > 0$. Thus, 
\[
\p{\mathcal{M}_{k,n}} = \Omega(\p{B_{k,n} \neq \emptyset}). 
\]
To prove the lemma it thus suffices to establish (\ref{bknbound}). By linearity of expectation we have 
\begin{align}
\p{B_{k,n} \neq \emptyset} 	
& = \sum_{T\subset \mathcal U_{\le n}} \probC{B_{k,n}\neq\emptyset}{\mathcal{T}_{\leq n}=T}\cdot \mu [T]_{\le n} \nonumber\\
& = \sum_{T\subset \mathcal U_{\le n}} \sum_{v \in T_n} \probC{v \in B_{k,n}}{\mathcal{T}_{\leq n}=T}\cdot \mu[T]_{\le n} \nonumber\\
& = (\e B)^n \sum_{T\subset \mathcal U_{\le n}} \sum_{v \in T_n}\probC{v \in B_{k,n}}{\mathcal{T}_{\leq n}=T} \cdot \hat{\mu}^\star[T,v]_{\le n} \nonumber\\
& = (\e B)^n \cdot \phatc{v_n \in \hat{B}_{k,n}}.\label{eq:tiltbkn}
\end{align} 
To bound $\phat{v_n \in B_{k,n}}$, we first introduce the concept of a {\em useful} walk. Let $m_0=m_0(c)$ be the constant whose existence is guaranteed by Corollary~\ref{gnk} (in the current setting, $c=\sqrt 2$).
We say that the random walk, $S_0, S_1, \dots, S_n$, is $(k,n)$-useful if 
$$S_{n} = k, \qquad 0 \leq S_{i} < k \quad \forall 0 < i \le n,\qquad \mbox{and}\qquad S_{n-m}<k-m^{1/7}\quad \forall m\ge m_0.$$  
We let $U_{k,n}$ be the event that the random walk leading to $v_n$ in $\hat{\mathscr{L}_n}$ is $(k,n)$-useful, 
and remark that $U_{k,n}$ is a precise analogue of the event in Corollary~\ref{gnk}. 
If $U_{k,n}$ occurs, then for $\{v_n \in \hat{B}_{k,n}\}$ to occur it suffices that the following events occur:
\[\hat{\mb}^{\star v_{n-i}} < i^{1/7}\quad \forall m_0 \le i \le n\qquad \mbox{and}\qquad \hat{\mb}^{\star v_{n-i}} \le 0 \quad \forall 0\le i< m_0.\]
We remark that the events $U_{k,n}$, $\{\hat{\mb}^{\star v_{n-i}} \le 0\}$, $i=1, \dots, m_0$, and $\{\hat{\mb}^{\star v_{n-i}} < i^{1/7}\}$, $i=m_0,\ldots,n$ are mutually independent. This holds as, first, $U_{k,n}$ depends only on 
edge weights on the path from the root to $v_n$, second each random variable $\hat{\mb}^{\star v_{n-i}}$ depends only on the subtree of $\hat{\mathcal {T}}$ 
leaving $v_{n-i}$ off this path, and third, these subtrees are disjoint for distinct $i$. 

Let $\ell \ge 1$ be large enough that the number of children $C_i$ of $v_i$ satisfies $\p{C_i\le \ell}\ge 1/2$. Then,
$$
\phatc{\hat R^{\star v_i}\le 0}\ge \frac 1 2 \cdot \Cphatc{\hat R^{\star v_i}\le 0}{C_i\le \ell} \ge \frac 1 2 \cdot \p{R\le 0}^\ell \cdot \p{X\le 0}^\ell.$$ 
The latter probability is at least $\epsilon>0$ by (\ref{eq:hu-shi}) and since $\e X< 0$, so 
\begin{equation}\label{eq:lower_rstar}
\phatc{\hat R^{\star v_{n-i}}\le 0,~\forall 0\le i=0< m_0}\ge \epsilon^{m_0}>0.
\end{equation}
The lower bound in the lemma then follows from the next two inequalities
$$
\phat{U_{k,n}} = \Omega\pran{\frac{k}{n^2} \frac{e^{- \lambda^\star k}}{(\e B)^n}}\qquad \mbox{and}\qquad
\phatc{\hat{\mb}^{\star v_{n-i}} < i^{1/7}~\forall m_0\le i\le n} = \Omega(1).
$$
The first equation is an immediate consequence of Corollary~\ref{gnk} and the exponential change of measure result Lemma \ref{lem:transfer}. 
The second equation follows from the observation that the random variables $\hat \mb^{\star v_i}$, $i\ge 0$, are i.i.d., and  Lemma~\ref{enin} below. 
From these two bounds, (\ref{eq:tiltbkn}) and (\ref{eq:lower_rstar}), we immediately obtain 
\[
\p{B_{k,n} \neq \emptyset} = \Omega\bigg( (\e B)^n \cdot \phat{U_{k,n}} \cdot \prod_{m_0\le i\le n} \phatc{\hat{\mb}^{\star v_{n-i}} < i^{1/7}}\bigg) = \Omega\pran{\frac{ke^{- \lambda^\star k}}{n^2}},
\]
proving the lemma.
\end{proof}
\begin{lem}\label{enin}Let $\hat \mb^{\star v_i}$ be as defined by (\ref{eq:def_rstar}). If $\E{B\log^8 B} < \infty$ then as $n\to\infty$, we have 
\[\phatc{\hat \mb^{\star v_{n-i}}< i^{1/7}~\forall 0\le i\le n} = \Omega(1).\]
\end{lem}
\begin{proof}
To shorten notation, we will in fact bound $\phatc{\hat \mb^{\star v_{i}}< i^{1/7}~\forall 0\le i\le n}$ -- 
which by symmetry is identical to the quantity we wish to bound. 
As noted above, the events $\{\hat \mb^{\star v_i}< i^{1/7}\}$ are independent for distinct $i$ since they 
depend on disjoint subtrees of $\hat{\mathscr{L}}$. Next, recall that $C_i$ is the number of children 
of $v_i$ aside from $v_{i+1}$ and is distributed as $\hat{B}-1$. Call these children $v_{i,1},\ldots,v_{i,C_i}$, let the displacement from 
$v_{i}$ to $v_{i,j}$ be $X_{i,j}$, and let the subtree of $\hat{\mathscr{L}}$ rooted at $v_{i,j}$ be $\mathscr{L}_{i,j}$ 
(note that it is distributed as $\mathscr L$). 

Now fix $i$. Then for $\{\hat \mb^{\star v_i}< i^{1/7}\}$ to occur, it suffices that for each $j=1,\ldots,C_{i}$, the following inequality holds $\mb^{v_{i,j}}+X_{i,j}\leq i^{1/7}$. 
We thus have, for independent $R$ and $X$, 
\begin{align}
\phatc{\hat \mb^{\star v_i}< i^{1/7}} 
& \geq \sum_{k=1}^{\infty} \p{C_{n-i}=k}\cdot \pc{\mb +X < i^{1/7}}^k \nonumber\\
& \geq \sum_{k=1}^{\infty} \frac{(k+1)\p{B=k+1}}{\e B} \cdot \pc{X <i^{1/7}/2}^k\cdot \pc{R < i^{1/7}/2}^k.\label{sumtoinfty}
\end{align}
                        
By Lemmas~\ref{lem:chernoff} and~\ref{maxupperbound}, respectively, applied to the first and second probabilities in the last line above, 
for some $c_1,c_2 > 0$ and all $i$ sufficiently large (say $i \geq i_0$) we obtain the following bound: 
\begin{align*}
\pc{X <i^{1/7}/2}\cdot \pc{R < i^{1/7}/2} 	
& \geq (1-c_1e^{-\lambda^\star i^{1/7}/2})\cdot (1-e^{-\lambda^\star i^{1/7}/2}) \\
& \geq 1-c_2e^{-\lambda^\star i^{1/7}} > 1/2. 
\end{align*}
Since when $x < 1/2$, $\log(1-x) > -2x$, we have, for $i\ge i_0$,
\begin{align*}
\pc{R < i^{1/7}/2}^k\cdot \pc{X <i^{1/7}/2}^k 	
& \geq \exp\big(-2kc_2 e^{-\lambda^\star i^{1/7}}\big)\\
& \geq 1-2kc_2 e^{-\lambda^\star i^{1/7}}\\
& \ge 1-i^{-2},
\end{align*}    
for all $k\le k^\star(i)=\lfloor e^{\lambda^\star i^{1/7}}/(2c_2i^2) \rfloor$.

Furthermore, since $\E{B\log^{8}B}<\infty$, for any integer $m\ge 2$ we have the bound
\begin{align*}
\p{C_{i} \geq m} = \sum_{k > m} \frac{k\p{B=k}}{\e B} 
& \le \sum_{k>m} k \p{B=k}\\
& < \frac{1}{\log^{8} m} \sum_{k > m} k \log^{8} k \cdot \p{B=k} \\
& < \frac{\E{B\log^{8}B}}{\log^{8} m}.
\end{align*}
So, truncating the sum in (\ref{sumtoinfty}) at $k^\star$, 
for $i$ large enough that 
\[
\log^{8}\pran{\frac{e^{\lambda^\star i^{1/7}}}{2c_3i^2}} \geq \frac{(\lambda^\star i^{1/7})^{8}}{2}, 
\]     
we obtain 
\begin{align*}
\phatc{\hat \mb^{\star v_i}< i^{1/7}} 
& \geq \sum_{k=1}^{k^\star} \frac{(k+1)\p{B=k+1}}{\e B} \cdot \pran{1-i^{-2}} \\
			& \geq \pran{1-\p{C_{i}>k^\star}}\cdot \pran{1-i^{-2}} \\
			& \geq \pran{1-\frac{2\E{B\log^{8}B}}{(\lambda^\star i^{1/7})^{8}}}\cdot \pran{1-i^{-2}} > \frac 1 2,
\end{align*}
the last inequality holding for $i$ sufficiently large (say $i \geq i_1$, for $i_1\ge i_0$ large enough). 
Taking a product over $i \geq i_1$ yields
\begin{equation}\label{bigi}
\prod_{i_1\le i\le n}\phatc{\hat R^{v_i}< i^{1/7}} \geq \prod_{i=i_1}^n \bigg(1-\frac{3\E{B\log^{8}B}}{(\lambda^\star i^{1/7})^{8}}\bigg) = \Omega(1). 
\end{equation}
For smaller values of $i$, fix $m\ge 1$ large enough that $\p{C_{i} \leq m} \geq 1/2$, and note that in order for 
$\hat \mb^{\star v_i}< i^{1/7}$ to occur it suffices that first, $C_{i} \leq m$ and second, $R^{v_{i,j}} + X_{i,j} \leq 0$ for each 
$j=1,\ldots,C_{i}$. Taking $\epsilon > 0$ small enough that $\p{R + X \leq 0} > 2\epsilon$, we then have 
\[
\phatc{\hat \mb^{\star v_i}< i^{1/7}} \geq \frac 1 2 \hat{\mathbf{P}}(\hat \mb^{\star v_i}< i^{1/7}\,|\, C_{i} \leq m) \geq \epsilon^m, 
\]
which implies that 
\[
\prod_{0\le i< i_1}\phatc{\hat \mb^{\star v_i}< i^{1/7}} \geq \epsilon^{m i_2} = \Omega(1).
\]
Combining this last equation with (\ref{bigi}) completes the proof.
\end{proof}

\section{The size of the progeny: Proof of Theorem~\ref{thm:critical}}\label{sec:ZlogZ}
%!TEX root = killedbrw.tex   

In this section, we prove our main result, Theorem~\ref{thm:critical}, using the analysis of the shape of random walks in Section~\ref{sec:ballot}. 

\begin{lem}\label{lem:finite-mean}For a well-controlled critical killed branching random walk, the total progeny $Z$ satisfies $\E Z <\infty$.
\end{lem}  
\begin{proof}Recall that $\mathscr L_n$ denotes the set of un-killed nodes $n$ levels away from the root. Using the size-biasing from (\ref{eq:size-bias}), we have
\begin{align*}
\E{Z} = \sum_{n\ge 0} \E{|\mathscr L_n|}=\sum_{n\ge 0} (\e B)^n \cdot \phatc{v_n\in \hat{\mathscr L}_n}
&=\sum_{n\ge 0} (\e B)^n \cdot \p{S_i\ge 0~ \forall 0\le i\le n}\\   
&=\sum_{n\ge 0} (\e B)^n \cdot \sum_{k\ge 0} \p{S_i\ge 0, 1\le i\le n; S_n=k}.                                
\end{align*}  
Since the killed branching random walk is critical, $e^{f(\lambda^\star)}=\e B$ and $\Lambda'(\lambda^\star)=0$.
Splitting the above sum at $k=\log^2 n$, using the transfer of the ballot result of Theorem~\ref{thm:ballot_mean0} into the regime of large deviations, and Chernoff's bound (Lemma~\ref{lem:chernoff}), we see that, for $n$ large enough, the terms of the series above satisfy
\begin{align*}
\E{|\mathscr L_n|}
&\le \e B)^n \cdot \sum_{k=0}^{\lfloor\log^2 n\rfloor}\p{S_i\ge 0, 0\le i\le n; S_n=k} + (\e B)^n \cdot \p{S_n \ge \log^2 n}\\
&\le  C_1 \sum_{k=0}^{\lfloor \log^2 n\rfloor} \frac{k e^{-\lambda^\star k}}{n^{3/2}} + e^{-\lambda^\star \log^2 n}\le C_2 n^{-3/2},
\end{align*}
for some constants $C_1,C_2$. It follows immediately that $\E{Z}<\infty$.
\end{proof}
        
We now use a similar lower bounding technique in order to complete the proof of Theorem~\ref{thm:critical}.

\begin{lem}\label{lem:zlogz}For a critical killed branching random walk, if 
$\E{B\log^8 B}< \infty$ then we have $\E{Z \log Z}=\infty$.
\end{lem}
\begin{proof}
When $M= k$ occurs, there is a node $u \in \mathscr L$ such that $S_u =k$. For each $u \in \mathcal U$, 
denote by $\mathcal M^u_k$ the event that $M=k$ and that additionally $u \in \mathscr L$, $S_u =k$, and $u$ is the lexicographically least node in $\mathscr L$ for which $S_u=k$ (for this we recall that the nodes of $\mathcal U$ are labelled by $\bigcup_{n=0}^{\infty} \N^n$, for our notion of lexicographic ordering). The events $\mathcal M^u_k$ are disjoint for distinct $u$ and $k$, and 
\[
\{M = k\} =\bigcup_{u \in \mathcal U} \mathcal M^u_k, 
\]
so, writing $Z_u=|\{v\in \mathscr L: u\preceq v\}|$ for the set of living nodes in the subtree rooted at $u$, we have
\begin{align}
\E{Z\log Z} 	
& = \sum_{k \geq 0} \sum_{u \in \mathcal U} \CExp{Z\log Z}{\mathcal M^u_k} \p{\mathcal M^u_k} \nonumber\\
& \geq \sum_{k\geq 0} \p{M=k}\cdot \inf_{u \in U} \CExp{Z\log Z}{\mathcal M^u_k} \nonumber \\
& \geq \sum_{k\geq 0} \p{M=k}\cdot \inf_{u \in U} \CExp{Z_u\log Z_u}{\mathcal M^u_k} \label{mainequation}
\end{align}
To bound $\CExp{Z_u}{M^u_k}$ we must first re-express the events $\mathcal M^u_k$. 
Let $\mathcal E^u_k$ be the event that $u$ is the lexicographically least node in $\mathscr L$ for which $S_u=k$. 
Also, write 
$M_u$ for $\max\{S_v~:~v \in \mathscr L, u \not \preceq v\}$. Recall that $M^u$ is the 
maximum position {\em relative} to $u$ in the subtree of $\mathscr L$ rooted at $u$. 
Now fix $u \in \mathcal U$ arbitrarily, and express the event $\mathcal M^u_k$ as follows. 
\[
\mathcal M^u_k = \{S_u = k\} \cap \{u \in \mathscr L\} \cap \{M^u=0\} \cap \mathcal E^u_k \cap \{M_u \leq k\}.
\]
Given that $S_u=k$, $u \in \mathscr L$, and $M^u=0$, the random variable $Z_v$ is independent of the events 
$\mathcal E^u_k$ and $\{M_u \leq k\}$. Thus, 
\begin{align*}
\CExp{Z_u}{\mathcal M^u_k} 	& = \CExp{Z_u}{S_u=k, u \in \mathscr L, M^u=0} \\
						& \geq \CExp{Z_u\I{M^u=0}}{S_u=k, u \in \mathscr L}.
\end{align*}
Now let  $R^u= \max\{ S_v-S_u~:~v \in \mathcal T, u \preceq v\}$, and note that if 
$R^u=0$ then certainly $M^u=0$, and so the previous equation gives 
\[
\CExp{Z_u}{\mathcal M^u_k} \geq \CExp{Z_u\I{R^u=0}}{S_u=k, u \in \mathscr L}.
\]
Given that $S_u=k$ and that $u \in \mathscr L$, $Z_u$ consists of all descendants $v \in \mathcal T$ with $u \preceq v$ such that for all $w$ with $u \preceq w \preceq v$, $S_w - S_u \geq -k$, so 
\begin{align}
\CExp{Z_u}{\mathcal M^u_k}	& \geq \E{|\{v \in \mathcal T~:~S_w \geq -k~\forall~w \preceq v\}| \I{R=0}} \nonumber \\
						& \geq \E{|\{v \in \mathcal T~:~S_v=-k,0 \geq S_w \geq -k~\forall~w \preceq v\}| \I{R=0}} \nonumber \\
						& \geq \sum_{k^2 \leq n \leq 2k^2} \E{|\{v \in \mathcal T_n~:~S_v=-k,S_w \geq -k~\forall~w \preceq v\}| \I{R=0}}.\label{almostthere}
\end{align}
By size-biasing, we have 
\begin{align}
&\E{|\{v \in \mathcal T_n~:~S_{v_n}=-k,0 \geq S_w \geq -k~\forall~w \preceq v_n\}| \I{R=0}} \nonumber\\
&\qquad\qquad = (\e B)^n \cdot \phat{S_{v_n}=-k,0 \geq S_w \geq -k~\forall~w \preceq v_n, R=0} \nonumber\\
&\qquad\qquad = (\e B)^n \cdot \phat{S_{v_n}=-k,0 \geq S_{v_i} \geq -k~\forall~0 \leq i \leq n} \nonumber\\
& \qquad\qquad\qquad \times	\Cphatc{R=0}{S_{v_n}=-k,0 \geq S_{v_i} \geq -k~\forall~0 \leq i \leq n}
\label{sizebiased}
\end{align}
By an argument just as that used in Lemma \ref{mknlowerbound}, it is straightforward to see that 
there is $\gamma_0 > 0$ such that for all $k$ sufficiently large and all $k^2 \leq n \leq 2k^2$, 
\[
\Cphatc{R=0}{S_{v_n}=-k,0 \geq S_{v_i} \geq -k~\forall~0 \leq i \leq n} \geq \gamma_0. 
\]
Also, by Corollary \ref{fnk} applied to the random walk $\{S_{n-i}-S_n\}_{0 \leq i \leq n}$, together with  
Lemma \ref{lem:transfer}, we obtain that 
\[
\phat{v_n \in \mathcal T_n,S_{v_n}=-k,0 \geq S_{v_i} \geq -k~\forall~0 \leq i \leq n} 
= \Theta\pran{\frac{1}{k^3} \frac{e^{\lambda^\star k}}{(\e B)^n}}, 
\]
and so combining (\ref{almostthere}) and (\ref{sizebiased}) with the two preceding equations,
it follows that there exist $\gamma_1 > 0$ and $K_1 \geq 0$ such that for all $k \geq K_0$, 
\[
\CExp{Z_u}{\mathcal M^u_k} \geq \gamma_1 \sum_{k^2 \leq n \leq 2k^2} \frac{e^{\lambda^\star k}}{k^3} 
						= \gamma_1 \frac{e^{\lambda^\star k}}{k}.
\]
By the conditional Jensen's inequality applied to the convex function $x\mapsto x\log x$, 
we then have that for some $\gamma_2 > 0$ and $K_2 \geq 0$, for all $k \geq K_2$ and all $u \in \mathcal U$, 
\[
\CExp{Z_u\log Z_u}{\mathcal M^u_k} \geq \CExp{Z_u}{\mathcal M^u_k} \log \CExp{Z_u}{\mathcal M^u_k} \geq \gamma_2 e^{\lambda^* k}, 
\]
so by (\ref{mainequation}) and Lemma~\ref{maxlowerbound}
\[
\E{Z} \geq \gamma_2 \sum_{k \geq K_2} e^{\lambda^* k} \p{M = k} = \Omega\Bigg(\sum_{k \geq K_2} \frac{1}{k}\Bigg)=\infty.\qedhere
\]
\end{proof}   

\section{Proofs of the ballot results}\label{sec:proofs-ballot}
%!TEX root = killedbrw.tex
We first state two basic lemmas that will be useful in the proof of Theorem~\ref{thm:gbt_new}. 
We include a proof of the first, simple result for completeness. 
\begin{lem}\label{lem:easy_donsker}
	If $\E{X^2}<\infty$ then for all $\alpha$ and $\beta$ with $0 < \alpha < \beta$, there is $\gamma > 0$ such that for all $n$ large enough 
	and all $n'$ with $0 < n' \leq n$, 
	\[
		\p{ |S_{n'}| \leq \alpha\sqrt{n}, \max_{1 \leq i \leq n} |S_i| \leq \beta \sqrt{n}} \geq \gamma.
	\]
\end{lem}
\begin{proof}
	We assume for simplicity that $\E{X^2}=1$. 
	Let $W$ be a standard Brownian motion. By Exercise III.3.15 in \cite{ReYo2004}, there is $\gamma > 0$ such that for all $n$ and $n'$ with $0 < n' \leq n$, 
	\[
		\p{\sup_{0 \leq t \leq {n'}} |W_t| \geq \frac{\beta+\alpha}{2}\cdot \sqrt{n}, |W_{n'}| \leq \frac{\alpha}{2}\cdot \sqrt{n} } \geq 2\gamma. 
	\]
	Furthermore, by Donsker's theorem \cite{Donsker1952} (see also \cite{RoWi2000}, I.8.3), the random walk $S$ can be embedded in $W$ such that for all $n$ large enough, 
	\[
		\p{\max_{1 \leq k \leq n} |S_k - W_k| \geq \frac{\min\{\alpha,\beta-\alpha\}}{2}\cdot \sqrt{n}} \leq \gamma. 
	\]
	Combining the two preceding bounds completes the proof. 
\end{proof}
We will also use the following lemma, Lemma 3.3 from \citep{pemantle95critical}:
\begin{lem}\label{lem:pem_per}
	Let $S_n$, $n\ge 0$ be a random walk with step $X$, $\E{X}=0$. For $h \geq 0$, let $N_h$ be the first time $n\ge 0$ that $S_n < -h$. Then for any $a > 0$ there are constants $c_1,c_2,c_3$ such that for all~$n$:
	\begin{compactenum}[\rm (a)]
		\item for all $h$ with $0 \leq h \leq a\sqrt{n}$, $\p{N_h \geq n} \geq c_1\cdot (h+1)/\sqrt{n}$; 
		\item for all $h$ with $0 \leq h \leq a\sqrt{n}$, $\CExp{S_n^2}{N_h > n} \leq c_2 n$; and 
		\item for all $h\geq 0$, $\p{N_h \geq n} \leq c_3 \cdot (h+1)/\sqrt{n}$.
	\end{compactenum}
\end{lem}
We additionally require the following uniform local limit theorem. This is a weakening of Theorem 1 from \cite{stone65local}. (See also \cite{petrov75sums}.) 
\begin{thm}[\cite{stone65local}]
\label{thm:stone}
Fix any $c > 0$ and a random variable $X$ with lattice $\Z$. If $\e{X}=0$ and $0 < \Ec{X^2} < \infty$ 
then for all integers $x$ with $|x| \leq c \sqrt{n}$, 
\[
\p{S_n = x} = (1+o(1))\frac{e^{-x^2/(2n\Ec{X^2})}}{\sqrt{2\pi\Ec{X^2}n}},
\]
where $o(1) \rightarrow 0$ as $n \rightarrow \infty$ uniformly over all $x$ in the allowed range. 
\end{thm}
Finally, a useful trick, both in proving Theorem~\ref{thm:gbt_new} and when applying the theorem and its corollaries, 
is to turn the random walk ``upside-down and backwards''. 
By this we mean that we consider the random walk $S^r$ with $S^r_0=0$ and, for $0 \leq i < n$, 
with 
\[
S^r_{i+1} = -(X_n+\ldots+X_{n-i}) = S^r_{i}-X_{n-i}. 
\]
We refer to $S^r$ as ``the reversed random walk''. 
\begin{proof}[Proof of Theorem~\ref{thm:gbt_new}] 
For simplicity, we assume that $\Ec{X^2}=1$. 
We also assume that $k \geq 0$, as the case $k < 0$ follows from the case $k \geq 0$ by considering the reversed random walk $S^r$. 
The upper bound of the theorem is immediate from Theorem~\ref{thm:ballot_mean0} as the requirements in Theorem~\ref{thm:gbt_new} are more restrictive. 
To prove the lower bound, first let $\delta = \min\{\epsilon/11,1/(5c_3)\}$, where $c_3$ is the constant from Lemma \ref{lem:pem_per}. 
Let $N_-$ be the first time $n\ge 0$ that $S_n \leq -m$, and let $N_+$ be the first time $n\ge 0$ that $S_n \ge 5\delta\sqrt{n}$. 
The events that $N_- \geq \lfloor n/4\rfloor$ and that $N_+ \leq \lfloor n/4\rfloor$ are increasing in the values of the random walk steps, so they are positively correlated and by FKG inequality \cite{Harris1960,FoKaGi1971,AlSp2008}
\begin{eqnarray*}
	\p{N_+ \leq \min\{\lfloor n/4\rfloor,N_-\}} 	& \geq & \p{N_+ \leq \lfloor n/4\rfloor,N_- \geq \lfloor n/4\rfloor} \\
									& \geq & \p{N_+ \leq \lfloor n/4\rfloor}\cdot\p{N_- \geq \lfloor n/4\rfloor}.
\end{eqnarray*}
By Lemma \ref{lem:pem_per}, 
\[
\p{N_- \geq \lfloor n/4\rfloor} \geq \frac{c_1(m+1)}{\sqrt{\lfloor n/4 \rfloor}} \geq \frac{2c_1(m+1)}{\sqrt{n}},
\]
and 
\[
\p{N_+ \leq \lfloor n/4\rfloor} \geq  1- \frac{c_3(2\delta\sqrt{n}+1)}{\sqrt{\lfloor n/4\rfloor}} \geq \frac{1}{2},
\]
for $n$ large enough. So for $n$ large enough, 
\begin{equation}\label{eq:gbt_new2_1}
\p{N_+ \leq \min\{\lfloor n/4\rfloor,N_-\}} \geq \frac{c_1(m+1)}{\sqrt{n}}. 
\end{equation}
Next, since $\E{|X|^3}<\infty$, we have $\p{X \geq t} = o(t^{-3})$, so by the union bound, for all $n$ large enough, 
\[
\p{\max_{1 \leq i \leq n} X_i \geq \delta \sqrt{n}} \leq \frac{c_1}{2\sqrt{n}}.
\]
It follows from this fact and 
(\ref{eq:gbt_new2_1}) that 
\begin{align}\label{eq:gbt_new2_2}
\p{N_+ \leq \min\{\lfloor n/4\rfloor,N_-\}, S_{N_+} \leq 6\delta\sqrt{n}} 
&\geq \p{N_+ \le \min\{\lfloor n/4\rfloor,N_-\}} - \p{\max_{1 \leq i \leq n} X_i > \delta\sqrt{n}}\nonumber\\
& \geq \frac{c_1(m+1)}{2\sqrt{n}}.
\end{align}
Applying Lemma \ref{lem:easy_donsker} to the random walk restarted at time $N_+$, we see that there is $\gamma_1 > 0$ such that for any fixed 
$\alpha$ with $1/4 < \alpha < 1/2$, and $n$ large enough
\[ 
\p{|S_{\lfloor \alpha n\rfloor}-S_{N_+}| \leq \delta \sqrt n, \max_{N_+ < i \leq \lfloor \alpha n\rfloor} |S_i - S_{N_+}| \leq 2\delta\sqrt n} \geq \gamma_1. 
\] 
From this fact and (\ref{eq:gbt_new2_2}), it follows by the strong Markov property and the fact that $8\delta < \epsilon$ that 
\begin{equation}\label{eq:gbt_new2_3} 
\p{k+4\delta \sqrt{n} \leq S_{\lfloor \alpha n\rfloor} \leq k+7\delta\sqrt{n}, -m \leq S_i < k+\epsilon\sqrt{n}~\forall~0 < i \leq \lfloor \alpha n\rfloor} \geq \frac{c_1\gamma_1(m+1)}{2\sqrt{n}}. 
\end{equation} 
To shorten coming formulas, let $E_1$ be the event whose probability is bounded in (\ref{eq:gbt_new2_3}). 

Next, let $S^r$ be the random walk with $S^r_0=0$ and, for $0 \leq i < n$, $S_{i+1}^r = S_i^r - X_{n-i}$. Just as 
we derived (\ref{eq:gbt_new2_3}), one can see that there is $\gamma_2 > 0$ such that for $n$ sufficiently large, 
\begin{equation}\label{eq:gbt_new2_4}
\p{4\delta \sqrt{n} \leq S^r_{\lfloor \alpha n\rfloor} \leq 7\delta\sqrt{n}, -(k+m)\leq S_i^r < \epsilon\sqrt{n}~\forall~0 < i \leq \lfloor \alpha n\rfloor} \geq \frac{c_1\gamma_2(k+m+1)}{2\sqrt{n}}.
\end{equation}
We denote by $E_2$ the event whose probability is bounded in (\ref{eq:gbt_new2_4}). Also, let 
$Y = S_{\lfloor \alpha n \rfloor} - S_{\lfloor \alpha n \rfloor}^r$, 
so that 
\[
S_n = Y+\sum_{i=\lfloor \alpha n \rfloor+1}^{\lceil(1-\alpha) n\rceil} X_i.
\]
Observe that, if $E_1\cap E_2$ occurs, necessarily $3\delta\sqrt{n} \leq k-Y \leq 3\delta\sqrt{n}$. To see this, note 
for example that $k-Y = -3\delta\sqrt{n}$ can only occur if $S_{\lfloor \alpha n\rfloor} = k+7\delta\sqrt{n}$ 
and $S^r_{\lfloor \alpha n\rfloor}=4\delta\sqrt{n}$. 

Let $q =\lceil (1-\alpha)n\rceil-\lfloor \alpha n \rfloor$, and for $1 \leq i \leq q$, let 
	\[
	L_i = S_{\lfloor \alpha n \rfloor+i}-S_{\lfloor \alpha n \rfloor}, \hspace{0.2cm} \mbox{and let} \hspace{0.2cm} R_i=S_{\lceil (1-\alpha)n\rceil-i}-S_{\lceil (1-\alpha)n\rceil},
	\] 
	so in particular $L_q=-R_q$. 
Given that $E_1$ and $E_2$ occur, in order that $S_n=k$, that $S_i \geq -m$ for all $i=1,\ldots,n$, 
and that $S_i < k+\epsilon\sqrt{n}$ for all $0 < i < n$, it suffices that 
\begin{itemize}
	\item $L_m = S_{\lceil (1-\alpha)n\rceil} - S_{\lfloor \alpha n \rfloor} = k-Y$ 
		(we call this event $E_3$), and 
	\item for all $i$ with $1 \leq i \leq q$, 
		\[
				\min\{k-Y,0\}-\delta\sqrt{n} \leq L_i < \max\{k-Y,0\} + \delta\sqrt{n}
		\] 
		(we call this event $E_4$).
\end{itemize}
 \begin{figure}[htb]
 \centering
 \includegraphics[width=0.65\textwidth]{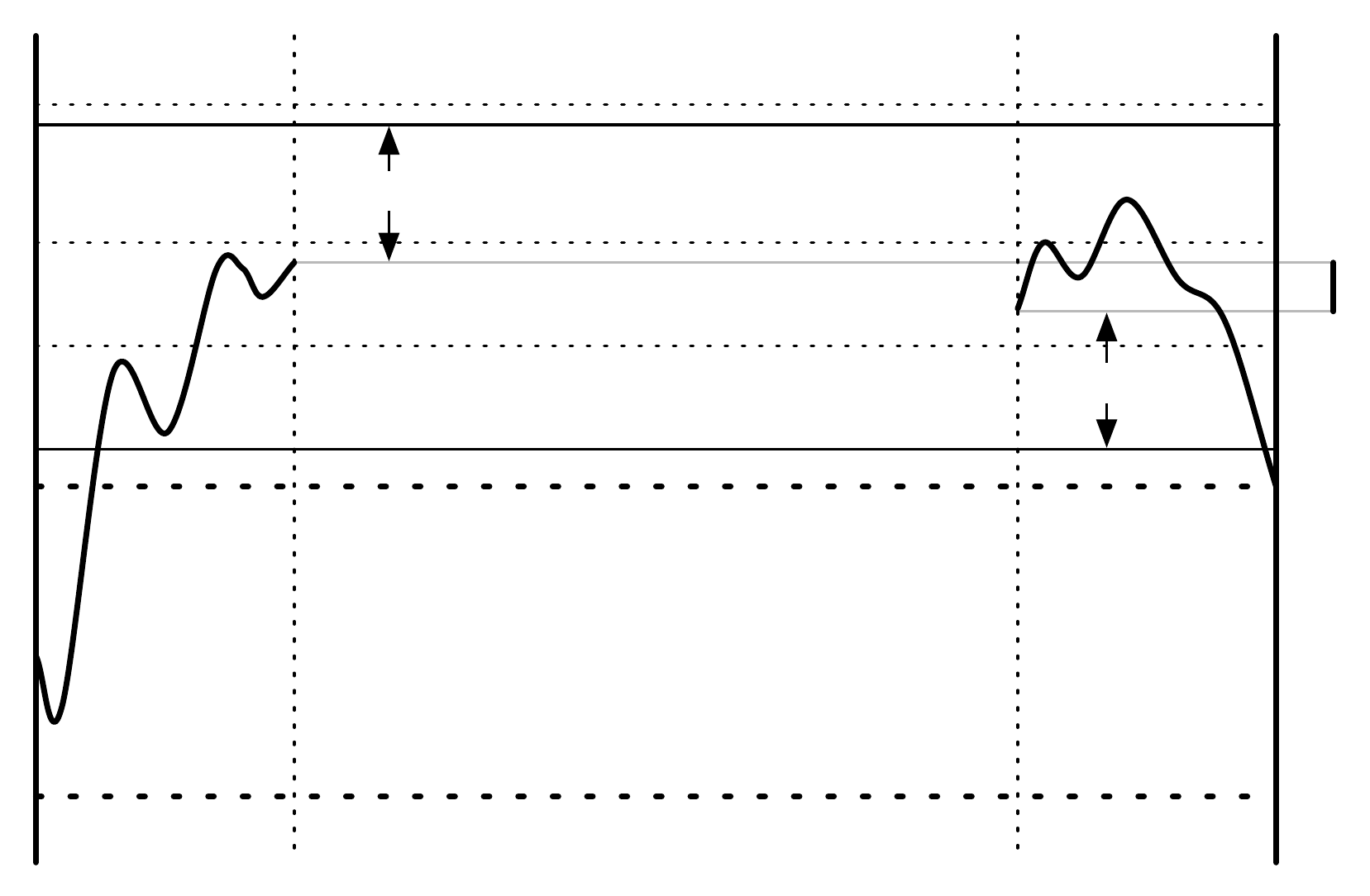}
\put(-279,82){$\scriptstyle{k}$}
\put(-288,18){$\scriptstyle{-m}$}
 \put(-3,121){$\scriptstyle{k-Y}$}
 \put(-305,131){$\scriptstyle{k+7\delta\sqrt{n}}$}
\put(-305,109){$\scriptstyle{k+4\delta\sqrt{n}}$}
\put(-301,158){$\scriptstyle{k+\epsilon\sqrt{n}}$}
\put(-64,103){$\scriptstyle{4\delta\sqrt{n}}$}
\put(-210,142){$\scriptstyle{4\delta\sqrt{n}}$}
\put(-272,-3){$\scriptstyle{0}$} %y-axis
\put(-279,46){$\scriptstyle{0}$} %x-axis
\put(-233,-3){$\scriptstyle{\lfloor\alpha n\rfloor}$}
\put(-105,-3){$\scriptstyle{\lceil(1-\alpha) n\rceil}$}
\put(-21,-3){$\scriptstyle{n}$}
 \caption{$E_3$ ensures that the middle portion of the random walk ``lines up'' with the outer portions. $E_4$ requires that the middle portion of the random walk stays between the two solid black horizontal lines. 
Given that $E_1$ and $E_2$ occur, these two horizontal lines must lie between $k$ and $k + \epsilon \sqrt{n}$, 
 so $E_4$ (more than) ensures that the middle portion of the random walk stays between $-m$ and $k + \epsilon \sqrt{n}$.}\label{fig:ballot}
 \end{figure}
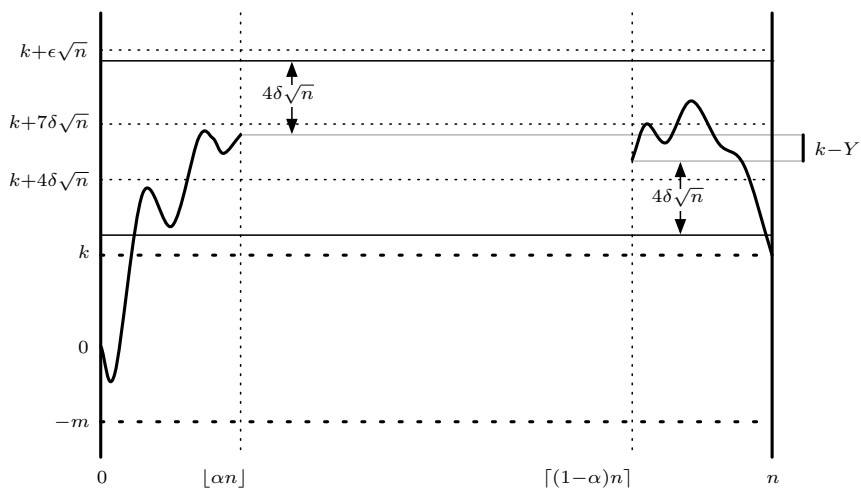
 These events are depicted in Figure~\ref{fig:ballot}. 
 By (\ref{eq:gbt_new2_3}) and (\ref{eq:gbt_new2_4}), to prove the lower bound it thus suffices to show that 
there is $\gamma_3>0$ not depending on $m$, $k$, or $n$ such that for all $n$ sufficiently large,
\begin{equation}\label{eq:gbt_new2_5}
	\probC{E_3,E_4}{E_1,E_2} \geq \frac{\gamma_3}{\sqrt{n}}. 
\end{equation}
Assuming that (\ref{eq:gbt_new2_5}) holds, since $E_1$ and $E_2$ are independent, 
combining (\ref{eq:gbt_new2_3}), (\ref{eq:gbt_new2_4}), and (\ref{eq:gbt_new2_5}) proves the claimed lower bound and completes the proof. We now turn to establishing (\ref{eq:gbt_new2_5}). 

	Let $\ssn=[-3\delta\sqrt{n},3\delta\sqrt{n}]\cap \Z$. Since $|k-Y| \leq 3\delta\sqrt{n}$, by the independence of disjoint sections of the random walk,  we then have 
	\begin{equation}\label{eq:gbt_new6} 
	\probC{E_3,E_4}{E_1,E_2} \geq \min_{p \in\ssn} \p{L_q= p, \min_{1 \leq i \leq q} L_i \geq -p^{-} -4\delta\sqrt{n}, \max_{1 \leq i \leq q} L_i \leq p^{+} +4\delta\sqrt{n}},
	\end{equation}
	where $p^- = -\min\{p,0\}$ and $p^+=\max\{p,0\}$. 
	Now fix $p \in \ssn$ arbitrarily.  
	For the remainder of the proof we assume that $p \geq 0$, 
	since the proof for the case $p < 0$ is obtained mechanically from the proof of the former by 
	reversing the roles of the random walks $L$ and $R$. Thus, the above probability becomes 
	\[
	\p{L_q= p, \min_{1 \leq i \leq q} L_i \geq -\delta\sqrt{n}, \max_{1 \leq i \leq q} L_i \leq p +\delta\sqrt{n}}.
	\]
	Let $B_p$ be the event that $\min_{1 \leq i \leq q} L_i \geq -4\delta\sqrt{n}$ and that 
		$\max_{1 \leq i \leq q} L_i \leq p +4\delta\sqrt{n}$. 
	We bound $\p{L_q=p,B_p}$ by first writing 
	\begin{equation}\label{eq:gbt_new7}
		\p{L_q=p,B_p} \geq \p{L_q=p}-\p{L_q=p,\overline{B_p}},
	\end{equation}
	where $\overline{B_p}$ denotes the complement of the event $B_p$.
	By Theorem \ref{thm:stone}, since $q=\Omega(n)$ and $p=O(\sqrt{n})$, 
	\begin{equation}\label{eq:gbt_new8}
		\p{L_q=p} = (1+o(1))\frac{e^{-p^2/(2 q)}}{\sqrt{2\pi q}} = \Omega\pran{\frac{1}{\sqrt n}},
	\end{equation}
	where $o(1)\rightarrow 0$ as $n \rightarrow \infty$, uniformly over all $p \in \ssn$ (recall that 
	we assume $\Ec{X^2}=1$). 
	To bound $\p{L_q=p,\overline{B_p}}$ from above, we first further divide the events $\{L_q=p\}$ and $B_p$. Let $q'=\lfloor n/2\rfloor - \lfloor \alpha n \rfloor$. Observe that 
	$\{L_q=p\}$ occurs if and only if $R_q = -p$. Similarly, if $\{L_q=p\}$ occurs, then for $\overline{B_p}$ to occur one of the following events must occur: either
	\begin{enumerate}
		\item $\min_{1 \leq i \leq q'} L_i < -4\delta\sqrt{n}$ (we call this event $C^{b}$); or 
		\item $\max_{1 \leq i \leq q'} L_i > p +4\delta\sqrt{n}$ (we call this event $C^{t}$); or 
		\item $\min_{1 \leq i \leq q-q'} R_i < -(p + 4\delta\sqrt{n})$ (we call this event $D^b$); or 
		\item $\max_{1 \leq i \leq q-q'} R_i > 4\delta\sqrt{n}$ (we call this event $D^t$). 
	\end{enumerate}
	Thus, 
	\begin{equation}\label{eq:gbt_new9}
		\p{L_q=p,\overline{B_p}} \leq \p{L_q=p,C^b} + \p{L_q=p,C^t} + \p{R_q=-p,D^b}+\p{R_q=-p,D^t}.
	\end{equation} 
	To complete the proof, it suffices to show that the sum on the right-hand side of (\ref{eq:gbt_new9}) 
	is at most $(1+o(1))e^{-p^2/(2 q)}/(2\sqrt{2\pi q})$, as 
	(\ref{eq:gbt_new6}) and (\ref{eq:gbt_new7}), and (\ref{eq:gbt_new8}) then imply that 
	$\probC{E_3,E_4}{E_1,E_2} = \Omega(1/\sqrt{n})$, as required. 
	We will show that each of the four terms on the right-hand side of (\ref{eq:gbt_new9}) is at most 
	$(1+o(1))e^{-p^2/(2 q)}/(8\sqrt{2\pi q})$, from which the required bound follows. 
	We provide all the details only for the bound on $\p{C^t, L_q=p}$, as the other bounds follow by rote 
	applications of the same technique. 
	
	Since $\Ec{|X|^3}< \infty$, we have $\p{\max_{1 \leq i \leq q'} |X_{\lfloor \alpha n \rfloor + i}| \geq \delta\sqrt{n}} = o(1/\sqrt{n})$, and so
	\begin{equation}\label{eq:gbt_new9b}
	\p{C^t,L_q=p} \leq \p{C^t,L_q=p,\max_{1 \leq i \leq q'} |X_{\lfloor \alpha n \rfloor + i}| < \delta\sqrt{n}} + o\pran{\frac{1}{\sqrt{n}}}.
	\end{equation}
	By Kolmogorov's maximal inequality \cite[See, e.g.,][]{Feller1968,petrov75sums}, 
	\begin{align}\label{eq:gbt_new10}
		\p{C^t,\max_{1 \leq i \leq q'} |X_{\lfloor \alpha n \rfloor + i}| < \delta\sqrt{n}} \leq \p{C^t} =
		&  \p{\max_{1 \leq i \leq q'} L_i > p +4\delta\sqrt{n}}\nonumber\\
		& \leq \frac{\Ec{L_q^2}}{(4\delta\sqrt{n})^2} \nonumber\\
		& = \frac{\Ec{X^2}\cdot q}{16\delta^2 n} \nonumber\\
		& \leq \frac{((1-2\alpha)n+1)}{16\delta^2 n} \leq \frac{1}{16}
	\end{align}
	for all $n$ sufficiently large, as long as we take $\alpha$ close enough to $1/2$ that 
	$(1-2\alpha) < \delta^2$.
	Furthermore, by the independence of disjoint sections of the random walk 
	and a simple conditioning, we have that 
	\begin{align*}
	\probC{L_q=p}{C^t,\max_{1 \leq i \leq q'} |X_{\lfloor \alpha n \rfloor + i}| < \delta\sqrt{n}} 
	& \leq  
	\hspace{-0.2cm}
	\mathop{\max_{1 \leq i \leq q'}}_{4\delta\sqrt{n} \leq j \leq 5\delta\sqrt{n}} 
	\hspace{-0.3cm}	
	\probC{L_q=p}{L_i=p+j,L_{i-1}\leq p+4\delta\sqrt{n} }. \\
	& = \mathop{\max_{1 \leq i \leq q'}}_{4\delta\sqrt{n} \leq j \leq 5\delta\sqrt{n}}  \p{S_{q-i} = -j}
	\end{align*}
	For any $i$ with $1 \leq i \leq q'$, we have 
	$q-i \geq q-q' \geq \lceil (1-\alpha)n \rceil - \lfloor n/2 \rfloor \geq q/2 = \Omega(n)$ and $j = O(\sqrt{n})$, 
	and it follows by Theorem~\ref{thm:stone} that 
	\begin{align*}
	\probC{L_q=p}{C^t,\max_{1 \leq i \leq q'} |X_{\lfloor \alpha n \rfloor + i}| < \delta\sqrt{n}} 
	& \leq (1+o(1))\max_{1 \leq i \leq q',4\delta\sqrt{n} \leq j \leq 5\delta\sqrt{n}} 
		\frac{e^{-j^2/(2(q-i))}}{\sqrt{2\pi(q-i)}} \\
	& \leq (1+o(1)) \frac{e^{-p^2/(2 q)}}{\sqrt{\pi  q }},
	\end{align*}
	the second inequality holding since $j \geq 4\delta\sqrt{n} > |p|$ and since $q-i \geq q/2$. 
	Combined with (\ref{eq:gbt_new9b}) and (\ref{eq:gbt_new10}), the latter inequality yields that 
	\[
	\p{C^t,L_q=p} \leq (1+o(1)) \frac{e^{-p^2/(2q)}}{16 \sqrt{\pi  q }} 
				\leq (1+o(1)) \frac{e^{-p^2/(2 q)}}{8\sqrt{2\pi  q }}. 
	\]
	An essentially identical proof shows that the same bound holds for $\p{C^b,L_q=p}$, and 
	the same reasoning applied to the reversed random walk $R$ shows that the same bound holds for 
	$\p{R_q=-p,D^b}$ and for $\p{R_q=-p,D^t}$. Combining these four bounds in 
	(\ref{eq:gbt_new9}) yields the required bound on $\p{L_q=p,\overline{B_p}}$ and completes the proof. 
\end{proof}

The following result strengthens Corollary \ref{fnk}; the strengthened version will be helpful in proving Corollary \ref{gnk}. Taking $m=0$ yields Corollary~\ref{fnk}.
\begin{lem}\label{fnkstrong}
Fix  $c > 1$. If $\E{|X|^3}<\infty$ Then under the conditions of Theorem \ref{thm:gbt_new}, 
for all $n$ and all $k$ and $m$ with $c^{-1} \leq k/\sqrt{n} \leq c $ and $0 \leq m \leq c\sqrt{n}$, 
\[
\p{S_n = k, -m \leq S_i < k~\forall~0 < i < n} = \Theta_{c}\pran{\frac{(m+1)(k+m+1)}{n^{2}}} = \Theta_{c}\pran{\frac{m+1}{n^{3/2}}}.
\]
\end{lem}

\begin{proof}
Define the {\em backwards} random walk $S^b$ by $S^b_0=0$ and for $i \geq 0$, $S^b_i = S^b_{i-1}+X_{n-i}$. 
In order that $S_n=k$ and that $-m \leq S_i \leq k$ for all 
$0 < i < n$, it is necessary and sufficient that for some integer $s$ with $-m \leq s \leq k$, we have 
\begin{itemize}
\item $S_{\lfloor n/2 \rfloor} = s$ and $-m \leq S_i \leq k$ for all $i$ with $0 \leq i \leq \lfloor n/2\rfloor$ (call this event $A_s$), and 
\item $S_{\lceil n/2 \rceil}^b = k-s$ and $0 \leq S_i^b \leq k+m$ for all $i$ with $0 \leq i \leq \lceil n/2\rceil$ (call this event $B_s$).  
\end{itemize}
The events $A_s$ and $B_s$ are independent and, for $s\neq s'$, $A_s$ and $A_{s'}$ are disjoint and $B_{s}$ and $B_{s'}$ are disjoint. 
Furthermore, for any $s$ with $k/3 \leq s \leq 2k/3$, 
\[
\min\{k-s,m+s\} \geq k/3 \geq \sqrt{n}/(3c) > \sqrt{\lceil n/2\rceil}/(3c),
\] 
so for such $s$ we can apply Theorem \ref{thm:gbt_new} 
with $\epsilon = 1/(3c)$ to bound $\p{A_s}$ and $\p{B_s}$. Since both $m+s$ and $k-s$ are $\Theta_{c}(n^{1/2})$ for all $s$ in the above range, we thus have 
\begin{eqnarray}
\p{S_n=k, -m \leq S_i \leq k~\forall~0 < i < n} & \geq& \sum_{k/3\le s \le 2k/3} \p{A_s,B_s} \nonumber\\
			& = & \Theta_{c}\Bigg(\sum_{k/3\le s \le 2k/3}\frac{(m+1)(m+s+1)}{n^{3/2}}\cdot\frac{k-s}{n^{3/2}}\Bigg)\nonumber\\
			& = & \Theta_{c}\pran{\frac{m+1}{n^{3/2}}},\label{eq:cor:gbt_new2_1}
\end{eqnarray}
proving the lower bound. To prove the upper bound, 
we observe that for any $s$ with $-m \leq s \leq k$, by dropping the condition that $S_i \leq k$ for $i$ from $1$ to $\lfloor n/2 \rfloor$ from the definition of $A_s$ we 
may use Theorem \ref{thm:ballot_mean0} to obtain the bound 
\[
\p{A_s} = O_{c}\pran{\frac{(m+1)(m+s+1)}{n^{3/2}}},
\]
and we may similarly see that $\p{B_s} = O_{c}((k-s)/n^{3/2})$. 
Summing these bounds over $-m \leq s \leq k$ yields the requisite upper bound. 
\end{proof}

Applying Theorem~\ref{thm:gbt_new} to the first $m$ steps of the random walk, 
and applying Lemma \ref{fnkstrong} to the random walk restarted at time $m$ yields the following corollary. This is straightforward and we omit the details.
\begin{cor}\label{fnksmj}
Fix $c > 0$. If $\E{|X|^3}< \infty$ 
then for all $n$, all $k$ with $c^{-1} < k/\sqrt{n} \le c$, all $1\le m \leq n/2$ and all $j \leq \min\{\sqrt{m},k/2\}$, 
\[
\p{S_n=k; 0\le S_i\le k~ \forall~ 0\le i<n;S_m = j} = \Theta\left(\frac{(j+1)^2
\cdot (k+1)}{m^{3/2}\cdot (n-m)^{2}}\right).
\]
\end{cor}
If $X$ satisfies the conditions of the above theorems and corollary, then so does $-X$, 
and so applying Corollary~\ref{fnksmj} to the reversed random walk $S^r$ and rewriting the result in terms of $S$, 
we obtain the following.
\begin{cor}\label{fnksmjrev}
Fix $c > 0$. If $\E{|X|^3}< \infty$ 
then for all $n$, all $k$ with $c^{-1} < k/\sqrt{n} \le c$, all $1\le m \leq n/2$ and all $j \leq \min\{\sqrt{m},k/2\}$, 
\[
\p{S_n=k; 0\le S_i< n~\forall 0\le i<n; S_{n-m} = k-j} = \Theta_c\left(\frac{(j+1)^2\cdot (k+1)}{m^{3/2}\cdot (n-m)^{2}}\right).
\]
\end{cor}
\begin{proof}[Proof of Corollary~\ref{gnk}]
First, let $F_{n,k}$ be the event that $S_n=k$ and $0\le S_i< k$ for all $i$ such that $0\le i<n$. Let also $H_{n,k}$ be the event that $F_{n,k}$ occurs, and $k/4\le S_{\lfloor n/2 \rfloor}\le k/2$.
For each integer $m$ with $k/4 \leq m \leq k/2$, we will consider the following two events: 
\begin{itemize}
\item $S_{\lfloor n/2\rfloor}=m$ and $S_i \leq 3k/4$ for all $0 \le i \le n/2$. (We call this event $A_m$.)
\item Writing $S^\star_i=S_{\lfloor n/2\rfloor+i}-S_{\lfloor n/2\rfloor}$, we have $S^\star_{n-\lfloor n/2\rfloor}=k-m$ and $-m \le S^\star_i \le k-m$ for 
$0 \le i \le n-\lfloor n/2 \rfloor$. (We call this event $B_m$.)
\end{itemize}
For all $n$ sufficiently large and for any $m$ in the above range, if $A_m$ and $B_m$ both occur then $H_{n,k}$ occurs.  
Now apply Theorem \ref{thm:gbt_new} to the event $A_m$, and Lemma \ref{fnkstrong} to the event $B_m$, 
and use the independence of $A_m$ and $B_m$ to see that 
\[
\p{A_m,B_m} = \p{A_m}\cdot \p{B_m} = \Theta_c\pran{\frac{m}{n^{3/2}} \cdot \frac{m k}{n^2}} = \Theta_c\pran{\frac{k}{n^{5/2}}},
\]
so summing over $k/4 \leq m \leq k/2$, we obtain that there is some constant $\gamma(c) > 0$ for which 
\begin{equation}\label{hnklower}
\p{H_{n,k}} \geq \frac{\gamma(c)k}{n^2}.
\end{equation}      
Next, for fixed integer $m_0>0$, let $B_{n,k}(m_0)$ be the event that there is $m\in [m_0,n/2]$ for which
$S_{n-m} > k-m^{1/7}$. By Corollary~\ref{fnksmjrev}, we have 
\begin{align*}
\p{F_{n,k} ;B_{n,k}(m_0)}	
& \leq \sum_{m=m_0}^{\lfloor n/2\rfloor} \sum_{j=1}^{\lfloor m^{1/7}\rfloor} \p{F_{n,k},S_{n-m} = k-j} \\
& = O_c\Bigg(\sum_{m=m_0}^{\lfloor n/2\rfloor} \sum_{j=1}^{\lfloor m^{1/7}\rfloor} \frac{(j+1)^2\cdot (k+1)}{m^{3/2}\cdot (n-m)^{2}}\Bigg) \\
& = O_c\Bigg(\sum_{m=m_0}^{\lfloor n/2\rfloor} \frac{k+1}{m^{15/14} \cdot n^2}\Bigg) \\
& = O_c\bigg(\frac{k+1}{m_0^{1/14} n^2}\bigg). 
\end{align*}
We may thus find $m_0=m_0(c)$ large enough that 
\[
\p{H_{n,k},B_{n,k}(m_0)}\le \p{F_{n,k},B_{n,k}(m_0)} \leq \frac{\gamma(c)k}{2n^2}.
\]
Combining this bound with (\ref{hnklower}), we obtain that 
\begin{equation} \label{eq:desired_lb}
\p{H_{n,k},\overline{B_{n,k}(m_0)}} \geq \frac{\gamma(c)(k+1)}{2n^2}.
\end{equation}        
The inequality $3k/4 \leq k-k^{1/7}$ clearly holds for all $k \geq c^{-1}\sqrt{n}$ as long as $n$ is sufficiently large. Thus, for $n$ sufficiently large, if $H_{n,k}$ and $B_{n,k}(m_0)^c$ both occur, then the desired event 
$$S_n=k, \qquad 0\le S_i<k\quad\forall~0\le i<n, \qquad\mbox{and}\qquad S_{n-i}\le k-i^{1/7}\quad\forall m_0\le i\le n$$  
also occurs, which by (\ref{eq:desired_lb}) yields the result. 
\end{proof}

\small
\setlength{\bibsep}{.3em}
\bibliographystyle{plainnat}
\bibliography{killedbrw}

\end{document}